\documentclass[12pt,twoside]{article}
\usepackage[latin1]{inputenc}
\usepackage{graphicx}
\usepackage{latexsym}

\usepackage{amsthm}
\usepackage{amssymb}
\usepackage{amsmath}
\usepackage{graphicx}
\usepackage{eufrak} 
\usepackage{eucal}  
\usepackage{a4}
\usepackage{comment}
\usepackage{epsfig}
\usepackage{hyperref}
\usepackage{xcolor}
\usepackage{float}
\usepackage{harvard}

\allowdisplaybreaks 
\newcommand{\vertgleich}{ \stackrel{\mathcal{D}}{=} }

\newcommand{\ceil}[1]{\lceil #1 \rceil}

\numberwithin{equation}{section}

\def\er{\mathbb{R}}

\def\leb{\mathbb{L}}

\def\cc{\mathcal{C}}

\def\e{\varepsilon}

\def\te{\vartheta}
\def\hte{\hat{\vartheta}}

\def\sjni{\sum_{i=1}^n}

\def\beq{\begin{eqnarray*}}
\def\eeq{\end{eqnarray*}}

\def\RMSE{\text{\rm RMSE}}

\def\my{Y^*_k}
\def\mY{\mathcal{Y}^*_k}

\def\y{\mathcal{Y}}
\def\fin{\Big(\frac in\Big)}

\input{mymath.sty}

\begin{document}

\title{Efficient estimation of functionals in nonparametric boundary models\footnote{We are grateful for very helpful comments and questions by the referees. Financial support by the DFG through Research Unit FOR 1735 {\it Structural Inference in Statistics: Adaptation and Efficiency} is  acknowledged.}
}

\author{\parbox[t]{6.5cm}{\centering Markus
 Rei\ss\\[2mm]
 \normalsize{\it
Institute of Mathematics\\
Humboldt-Universit\"at zu Berlin}\\
\mbox{} mreiss@mathematik.hu-berlin.de} and  \parbox[t]{6.5cm}{\centering Leonie Selk\\[2mm]
 \normalsize{\it
Institute of Mathematics\\
Universit\"at Hamburg}\\
\mbox{} leonie.selk@math.uni-hamburg.de}}

\maketitle

\begin{abstract}
For nonparametric regression with one-sided errors and a boundary curve model for Poisson point processes we consider the problem of efficient estimation for linear functionals. The minimax optimal rate is obtained by an unbiased estimation method which nevertheless depends on a H\"older condition or monotonicity assumption for the underlying regression or boundary function.

We first construct a simple blockwise estimator and then build up a nonparametric maximum-likelihood approach for exponential noise variables and the point process model. In that approach also non-asymptotic efficiency is obtained (UMVU: uniformly minimum variance among all unbiased estimators).The proofs rely essentially on martingale stopping arguments for counting processes and the point process geometry. The estimators are easily computable and a small simulation study confirms their applicability.
\end{abstract}

{\small \noindent {\it Key words and Phrases:} frontier estimation, support estimation, Poisson point process, sufficiency, completeness, UMVU, nonparametric MLE, shape constraint, monotone boundary, optional stopping.

\section{Introduction}

For   regression models
\begin{equation}\label{EqRegr}
 \y_i=g(i/n)+\eps_i,\quad i=1,\ldots,n,
 \end{equation}
the estimation of linear functionals of the  regression function $g$ is well understood if $(\eps_i)$ are uncorrelated with mean zero and variance $\sigma^2>0$.  Then the discrete functionals \begin{equation}\label{Eqthetan}
\theta^{(n)}=\frac1n\sum_{i=1}^ng(i/n)w(i/n) \text{ for some function } w:[0,1]\to\R
 \end{equation}
can be estimated by the plug-in version $\hat\theta_n=\frac1n\sum_{i=1}^n\y_iw(i/n)$ without bias and with variance $\frac{\sigma^2}{n^2}\sum_{i=1}^n w(i/n)^2$. By the Gau{\ss}-Markov theorem $\hat\theta_n$ has minimal variance among all linear and unbiased estimators. In the Gaussian case  $\hat\theta_n$ is even UMVU (uniformly of minimum variance among all unbiased estimators).
In the corresponding continuous-time signal-in-white-noise model
$dY(t)=g(t)dt+\sigma n^{-1/2}dW_t$, $t\in[0,1]$,
with a Brownian motion $W$ and some $g\in L^2([0,1])$ the plug-in estimator $\hat\theta=\int_0^1 w(t)dY(t)$ is equally an unbiased estimator of
\begin{equation}\label{Eqtheta}
\theta=\int_0^1 g(t)w(t)\,dt\text{ for some } w\in L^2([0,1])
 \end{equation}
 of variance $\frac{\sigma^2}{n}\int_0^1 w(t)^2dt$. By the Riesz representation theorem, we can thus estimate any linear $L^2$-continuous functional of $g$ with parametric rate $n^{-1/2}$.

In certain applications, however, the function $g$ is determined as the boundary or frontier function of the observations, which can be modeled equivalently by one-sided errors $(\eps_i)$.  The prototypical case is that $(\eps_i)$ are  i.i.d. with $\eps_i\ge 0$ and for some $\lambda>0$
\begin{equation}\label{EqPeps}
P(\eps_i\le x)=\lambda x+O(x^2) \text{ as } x\downarrow 0,
\end{equation}
e.g. $\eps_i\sim\Exp(\lambda)$. In that case the parametric rate for the location model (i.e. assuming $g$ to be constant) is with $n^{-1}$ much faster than in the regular case. These irregular statistical models have also found considerable theoretical interest, e.g. in the recent work by \cit{baraudbirge}. A rate-optimal estimator is given by the extreme value statistics $\min_i\y_i$. For the nonparametric problem of estimating the  function $g$ in $L^2$-loss, the optimal rate is $n^{-\beta/(\beta+1)}$ for $g$ in a H\"older ball of regularity $\beta\in(0,1]$ and radius $R>0$:
\begin{equation}\label{EqHoelder}
g\in \cc^{\beta}(R)=\Big\{f:[0,1]\to\R\,|\,\forall x,y\in[0,1]:\;\abs{f(y)-f(x)}\le R\abs{y-x}^{\beta}\Big\}.
\end{equation}
This is achieved by a local polynomial estimator $\hat g_{n,h}$ as in the regular case, see e.g. \cit{jirakmeisterreiss} for a construction and a survey of the large literature on that topic. A plug-in estimator $\hat\theta_{n}:=\int_0^1 \hat g_{n,h_n}(x)w(x)dx$ with optimal bandwidth $h_n$ to estimate $\theta$ in \eqref{Eqtheta} can achieve at best the rate $n^{-\beta/(\beta+1/2)}$. This is due to a pointwise bias of order $h^\beta$ and a pointwise variance of order $(nh)^{-2}$, which after integration and by using independence results in a total mean squared error of order $h^{2\beta}+n^{-2}h^{-1}$ for the plug-in estimator.  The standardised rate $n^{-\beta/(\beta+1/2)}$ is not optimal and for $\beta<1/2$ even slower than $n^{-1/2}$ in the regular case. At the heart of the problem is the usual nonparametric bias bound, which cannot be improved by averaging.

Here we show that the optimal estimation rate for $\theta^{(n)}$ under one-sided errors is $n^{-(\beta+1/2)/(\beta+1)}$ for $g\in\cc^\beta(R)$. The  improvement over the plug-in estimator is achieved by an unbiased estimation procedure. The bias is exactly zero for the  case of exponentially distributed errors and it is asymptotically negligible under \eqref{EqPeps} for $\beta>1/2$. Compared with standard nonparametric results it is remarkable that an unbiased estimator can be constructed whose rate is nevertheless worse than the parametric rate ($n^{-1}$ in this case). The risk bound comes from a trade-off between two terms in the variance instead of the usual bias-variance trade-off.

As for mean regression with the signal-in-white-noise model, also for one-sided errors an analogous continuous-time model is most useful in exhibiting the main statistical structure. It is given by observing a Poisson point process (PPP) on $[0,1]\times\R$ of intensity
\begin{equation}\label{EqPPP}
\lambda_g(x,y)=n{\bf 1}(y\ge g(x)),\quad x\in[0,1],\,y\in\R,
 \end{equation}
see e.g. \cit{karr} or \cit{daleyverejones} for point process properties and Figure 1 below for an illustration. For sufficiently regular $g$ this model can be shown to be asymptotically equivalent to the regression-type model \eqref{EqRegr} with $\lambda=1$ in \eqref{EqPeps}, cf. \cit{meister:reiss}. At the same time, this serves as a canonical model for support boundary estimation from i.i.d. observations. For instance, \cit{girardjacob} propose projection based estimators for $g$ in this model class and derive convergence rates as well as limit distributions, already relying on  bias reduction techniques.  Also \cit{bibingerjirakreiss} use it as an agnostic model for limit order books in financial markets.

We first develop the methods in the fundamental PPP model for $\theta$ from \eqref{Eqtheta} and then transfer them explicitly to the discrete model \eqref{EqRegr}.
By a blocking technique $\theta$ can be estimated without bias and at the minimax optimal rate. The method is then extended to the one-sided regression setting. Using Lepski's method we are then able to provide also an adaptive estimator, that is an estimator which does not rely on the smoothness parameters $\beta,R$ and still attains the minimax rate up to a logarithmic factor. In a second step we can even construct an estimator of $\theta$ which is UMVU. This non-asymptotic efficiency result is based on a nonparametric maximum-likelihood approach, where the maximum-likelihood estimator (MLE) $\hat g^{MLE}$ is not only explicit, but also forms a sufficient and complete statistics. In parallel with Gaussian mean regression we thus have the UMVU-property of the estimator, but  its asymptotic rate is worse than for parametric location estimation. Still, we are able to prove its asymptotic normality and to provide a self-normalising CLT such that asymptotic inference is feasible. The MLE approach equally works  for the class of monotone functions $g$.

The regression-type model \eqref{EqRegr} with one-sided errors and the PPP model \eqref{EqPPP} have a similar structure as models considered for density support estimation or image boundary recovery problems. Let us review briefly the literature on functional estimation for these statistical models. Many asymptotic results for the expected area of the convex hull for i.i.d. observations are based on the classical results by \cit{renyi:sulanke:1964}. Based on these results, the ideas of the present paper have been used by \cit{baldinreiss} to construct an UMVU estimator for the volume of a convex body. For image recovery problems \cit{korost:tsyb:book} describe already the rate $n^{-(\beta+1/2)/(\beta+1)}$ obtained for the functional $\int_0^1 g(x)dx$. The upper bound is based on a localisation step and loses a logarithmic factor. By threefold sample splitting \cit{gayraud} has constructed an estimator achieving this rate exactly for the related density support area estimation. An interesting linear programming approach is proposed by \cit{girardiouditskynazin}. Yet, these estimators are analysed asymptotically and lack the non-asymptotic unbiasedness and UMVU property we have found here.
Many other estimators are concerned with the estimation of the density support set or the regression-type function itself, not of the area or other functionals, let us mention the work by \cit{mammen:tsyb} for connections to classification problems. Specifically, a nonparametric MLE approach under monotonicity has been developed by \cit{korost:etal} for the asymptotically exact risk in estimating the density support set in Hausdorff distance. In Gaussian mean regression a nonparametric MLE over regular function classes is equivalent to a least-squares approach with roughness penalty, leading e.g. to smoothing splines. Under shape constraints the MLE is a well studied object, see e.g. \cit{groeneboomwellner}, but usually results are derived asymptotically.

In the next section we shall develop a simple block-wise estimator. Based on optional stopping for an
intrinsic martingale we prove that it is unbiased under the PPP model and under exponential noise in the regression-type model. For more general regression noise the required compensation cannot be achieved exactly, but it comes close to the model with corresponding exponential noise. The third part of that section presents the adaptive estimator, while the final part presents the lower bound implying that the rate is indeed optimal. The nonparametric MLE approach is presented in Section 3, first for the class of H\"older functions, then for monotone functions. The derivation of the completeness of the nonparametric MLE and the stopping arguments for the intrinsic martingale are intriguing. For the MLE under Lipschitz conditions we obtain central limit theorems which allow for feasible confidence sets.

In Section 4 we discuss some implications of the results, in particular concerning  estimating coefficients in a  projection estimator approach. Extensions and limitations are mentioned and a small simulation study shows that the estimators are numerically feasible and have satisfying finite-sample properties.  Most proofs are instructive and reveal some beautiful interplay between statistics, probability and geometry such that in the Appendix we only provide some technical lemmata (some of independent interest) and the more involved proofs of the adaptive rate and the CLT.
The notation follows the usual conventions. We write $a_n\lesssim b_n$ or $a_n=O(b_n)$ to say that $a_n$ is bounded by a constant multiple of $b_n$ and $a_n\thicksim b_n$ for $a_n\lesssim b_n$ as well as $b_n\lesssim a_n$. Moreover, $a_n=o(b_n)$ means $a_n/b_n\to 0$ and $a_n\asymp b_n$ stands for $a_n/b_n\to 1$.

\section{Simple rate-optimal estimation}

\subsection{Block-wise estimation in the PPP model}\label{ppp}

Let $(X_j,Y_j)_{j\ge 1}$ denote the observations of the Poisson point process (PPP) with intensity \eqref{EqPPP}. We shall estimate $\theta$ from \eqref{Eqtheta} without any bias. To grasp the main idea, suppose that $w(x)=1$ holds and that we know a deterministic function $\bar g:[0,1]\to\R$ with the property $\bar g\ge g$ (pointwise). Then the number of PPP observations below the graph of $\bar g$ is Poisson-distributed with intensity equal to $n$ times the area between $g$ and $\bar g$:
\[ \sum_{j\ge 1}{\bf 1}(Y_j\le \bar g(X_j))\sim \Poiss\Big(n\int_0^1(\bar g-g)(x)\,dx\Big).\]
This yields an unbiased pseudo-estimator $\bar\theta$:
\[ \bar\theta:=\int_0^1\bar g(x)\,dx-\frac1n\sum_{i\ge 1} {\bf 1}(Y_i\le \bar g(X_i)) \Rightarrow \E[\bar\theta]=\int_0^1 (\bar g-(\bar g-g))(x)\,dx=\theta.\]
The larger the area  between the graphs, the larger is the Poisson parameter and thus the variance of $\bar\theta$.

Now, we shall define an empirical substitute for
$\bar g$, which by stopping time arguments keeps the unbiasedness, but is nevertheless sufficiently close to $g$.
We partition $[0,1]$ in subintervals $I_k=[kh,(k+1)h)$ of length $h$ with $h^{-1}\in\N$ and note that the block-wise minimum $\my:=\min_{j:X_j\in I_k}Y_j$ satisfies $\my\ge\min_{x\in I_k}g(x)$. By the H\"older property of $g$ we conclude that $g(x)\le \my+Rh^\beta$ holds for all $x\in I_k$ and thus $\my+Rh^\beta$ is a local upper bound for $g$,  see also Figure \ref{fig-est-block}.
We thus estimate the functional locally on these blocks by
$$\hte_k:=(\my+Rh^\beta)\bar w_k-\frac{1}{nh}\sum_{i\ge 1}{\bf 1}\Big(X_i\in I_k,Y_i\le \my+Rh^\beta\Big)w(X_i),$$
where $\bar w_k=\frac1h\int_{I_k}w(x)dx$ and the true local parameter is $\te_k:=\frac 1h\int_{I_k}g(x)w(x)dx$.

\begin{figure}[t]\begin{center}
\begin{minipage}[t]{0.89\textwidth}
\includegraphics[width=\textwidth]{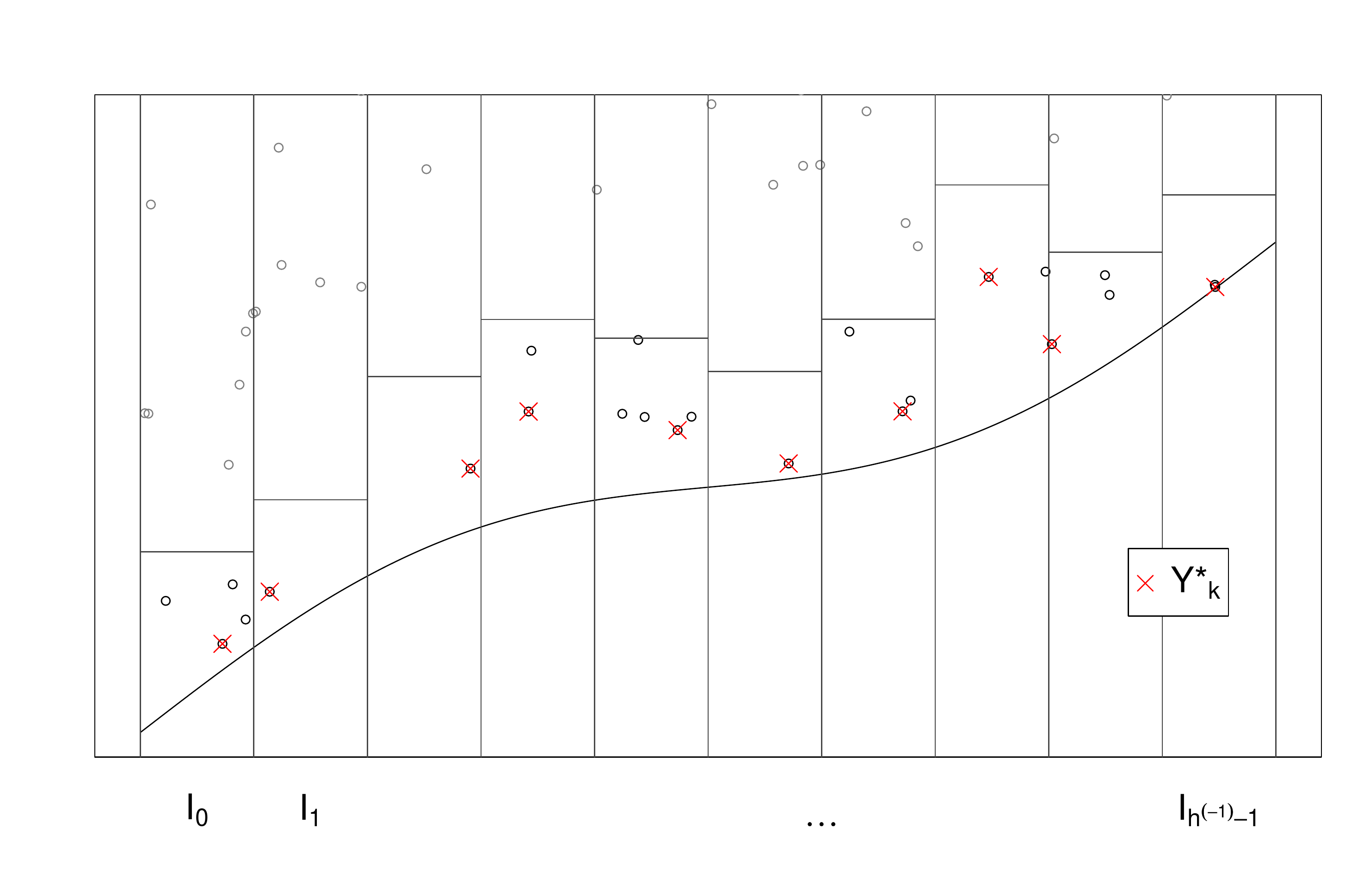}
\end{minipage}
\end{center}\vspace{-1cm}
\caption{Construction of  $\hte^{block}$. Circles indicate PPP observations $(X_i,Y_i)$, crosses blockwise minima $Y_k^\ast$ and horizontal lines the upper boundaries $Y_k^\ast+R h^\beta$.}\label{fig-est-block}
\end{figure}

\begin{theorem}
The estimator $\hte^{block}=\sum_{k=0}^{h^{-1}-1}\hte_k h$ satisfies with $\norm{w}_{L^2}^2=\int_0^1w(x)^2dx$
\[ \E[\hte^{block}]=\theta,\quad \Var(\hte^{block})\le \frac{2Rh^\beta+(nh)^{-1}}{n}\norm{w}_{L^2}^2.\]
In particular, the asymptotically optimal block size $h\asymp(2\beta Rn)^{-1/(\beta+1)}$ yields
\[\limsup_{n\to\infty}\sup_{g\in\cc^\beta(R)}n^{(2\beta+1)/(\beta+1)}\Var(\hte^{block})\le \frac{\beta+1}{\beta}(2\beta R)^{1/(\beta+1)}\norm{w}_{L^2}^{2}.\]
\end{theorem}

\begin{proof}

Let us study  the weighted counting process
\[ N(t):=\sum_{i\ge 1} {\bf 1}\Big(X_i\in I_k,Y_i\le t\Big)w(X_i),\quad t\in\R.\]
The pure counting process $\sum_i{\bf 1}(X_i\in I_k,Y_i\le t)$ is a point process in $t$ with deterministic intensity $\lambda_t=n\int_{I_k}(t-g(x))_+dx$. Hence, $(N(t),t\in\R)$ is a process with independent increments satisfying (e.g. via Prop. 2.32 in \cit{karr})
\[ \E[N(t)]=\int_{I_k}n(t-g(x))_+w(x)dx,\quad \Var(N(t))=\int_{I_k}n(t-g(x))_+w(x)^2dx.\]
In particular, $M(t)=N(t)-\E[N(t)]$ is a c\` adl\`ag martingale with respect to the filtration
\begin{equation}\label{EqFt}
{\cal F}_t=\sigma((X_i,Y_i){\bf 1}(Y_i\le t),\,i\ge 1),\quad t\in\R,
\end{equation}
 with mean zero and predictable quadratic variation $\langle M\rangle_t=\Var(N(t))$.

Now note that $\tau:=\my+Rh^\beta$ is an $({\cal F}_t)$-stopping time with
\begin{align}
P(\tau\ge t)&=\exp\Big(-n\int_{I_k}(t-Rh^\beta-g(x))_+dx\Big)\nonumber\\
&\le \exp\Big(-nh(t-\max_{x\in I_k}g(x)-R h^\beta)\Big)\label{tauest}
\end{align}
for $t\ge \max_{x\in I_k}g(x)+R h^\beta$. In particular, $\tau$ has finite expectation and  Lemma \ref{LemStop} on optional stopping yields
\[ \E[M(\tau)]=0 \Rightarrow \E[N(\tau)]=
n\int_{I_k}\E[(\tau-g(x))_+]w(x)dx\]
and
\[\Var(M(\tau))=\E[\langle M\rangle_\tau]=n\int_{I_k}\E[(\tau-g(x))_+]w(x)^2dx.\]
Noting $\tau\ge g(x)$  we have
\[ \E[(\tau-g(x))_+]=\E[\my]+Rh^\beta-g(x)\text{ for all $x\in I_k$}.\]
The identity
\[ \hte_k=\tau\bar w_k-\frac1{nh}N(\tau)=\theta_k-\frac1{nh}M(\tau)\]
 implies $\E[\hte_k]=\te_k$ and
\[ \Var(\hte_k)=\frac1{n^2h^2}\Var(M(\tau))=\frac{1}{nh^2}\int_{I_k}\E[\my+Rh^\beta-g(x)]w(x)^2dx.\]
A rough universal bound, using that $Y_k^\ast-\max_{x\in I_k}g(x)$ is stochastically smaller than the minimum in $y$ of a PPP with intensity $n{\bf 1}(x\in I_k,y\ge 0)$, yields with a random variable $E\sim \Exp(nh)$
\[\E[\my]\le \E\Big[\max_{x\in I_k}g(x)+E\Big]\le g(x)+Rh^\beta+(nh)^{-1}.\]
This implies
\[\Var(\hte_k)\le \frac{2Rh^\beta+(nh)^{-1}}{nh^2}\int_{I_k}w(x)^2dx.\]
We conclude for the final estimator $\hte^{block}=\sum_{k=0}^{h^{-1}-1}\hte_kh$
by the independence of $(\hte_k)_k$ that
\[ \E[\hte^{block}]=\te,\quad \Var(\hte^{block})\le \frac{2Rh^\beta+(nh)^{-1}}{n}\int_0^1w(x)^2dx.\]
Finally, insertion of the asymptotically optimal $h$ yields the  variance bound.
\end{proof}

\subsection{Blockwise estimation in the regression-type model}\label{regr}

We consider the equi-distant regression model \eqref{EqRegr}
 where $(\eps_i)$ are i.i.d. satisfying  \eqref{EqPeps}. The primary example will be $\eps_i\sim \Exp(\lambda)$, but any distribution on $\R^+$ with a Lipschitz continuous density $f_\eps$ at zero  and $f_\eps(0)=\lambda$ will be covered, as soon as some loose tail bound at infinity holds.

Since the observation design is discrete, our parameter of interest becomes
$\theta^{(n)}$ from \eqref{Eqthetan}. In analogy with the PPP case we build an estimator for $\theta_k=\frac1{nh}\sum_{i\in\tilde I_k}g(i/n)w(i/n)$ on each block of indices $\tilde I_k:=\{i:kh<\frac in\leq (k+1)h\}$, where $h^{-1},nh\in\N$:
\begin{equation}\label{EqTildethetak}\tilde\theta_k:=\frac{1}{nh}\sum_{i\in\tilde I_k}\Big(\y_i\wedge(\mY+Rh^\beta)-\lambda^{-1}{\bf 1}\Big(\y_i\le \mY+Rh^\beta\Big)\Big)w(i/n).
\end{equation}
Here, $\mY=\min_{i\in\tilde I_k}\y_i$ is again the minimal observation on each block. In contrast to the PPP-estimator the empirical upper bound for $g$ on $I_k$ is given by the minimum of $\mY+Rh^\beta$ and $\y_i$, which for the rate-optimal choice of $h$, however, has negligible impact. We obtain the following result where $\norm{w}_p=(\frac1n\sum_{i=1}^n\abs{w(i/n)}^p)^{1/p}$ denotes the standardised $\ell^p$-norm.

\begin{theorem}\label{ThmRegr}
Let the i.i.d. error variables $\eps_i$ satisfy \eqref{EqPeps} as well as $\bar F_\eps(y)=P(\eps_i>y)\lesssim (1+y)^{-\rho}$ for some $\rho>0$.
For $g\in\cc^\beta(R)$ the estimator $\tilde\theta_n^{block}=\sum_{k=0}^{h^{-1}-1}\tilde\theta_kh$ satisfies for $h\to 0$ with $nh\to\infty$ uniformly in $n,h,R,\beta$
\[ \abs{\E[\tilde\theta_n^{block}-\theta^{(n)}]}\lesssim (Rh^\beta+(n h)^{-1})^2 \norm{w}_1,\quad
\Var(\tilde\theta_n^{block})\lesssim h(Rh^\beta+(n h)^{-1})^2\norm{w}_2^2.\]
In particular, uniformly over $\beta\ge\beta_0>1/2$,  $R\le R_0<\infty$ we obtain with the rate-optimal block size $h\thicksim (Rn)^{-1/(\beta+1)}$
\[ (\E[\tilde\theta_n^{block}-\theta^{(n)}])^2=o(\Var(\tilde\theta_n^{block})),\quad
\Var(\tilde\theta_n^{block})\lesssim  R^{1/(\beta+1)}n^{-(2\beta+1)/(\beta+1)}\norm{w}_2^2.
\]
In the case $\eps_i\sim\Exp(\lambda)$ we have for any $\beta\in(0,1]$, $R,\lambda>0$ the more precise result
\[ \E[\tilde\theta_n^{block}]=\theta,\quad
\Var(\tilde\theta_n^{block})\le \frac{2Rh^\beta+(n\lambda h)^{-1}}{ n\lambda}\norm{w}_2^2.\]
\end{theorem}

\begin{remark}
The result and proof for $\eps_i\sim\Exp(\lambda)$ are exactly as in the PPP model. For other distributions of $\eps_i$ the estimator is only asymptotically unbiased, but for H\"older regularity $\beta>1/2$ the bias is negligible with respect to the stochastic error. A side remark is that for strong asymptotic equivalence with the PPP model in Le Cam's sense the necessary minimal regularity $\beta>1$ is higher, see \cit{meister:reiss}.
\end{remark}

\begin{proof}
Fix a block with index $k$ and consider for $t\in\R$
\begin{equation}\label{EqMartRegr}
 M(t):=\sum_{i\in \tilde I_k}\Big({\bf 1}(\y_i\le t)+\log\bar F_\eps(\y_i\wedge t-g(i/n))\Big)w(i/n).
 \end{equation}
With respect to the filtration ${\cal F}_t=\sigma(\y_i{\bf 1}(\y_i\le t),\,i\in\tilde I_k)$, $(M(t),t\in\R)$ defines a martingale with $\E[M(t)]=0$ and quadratic variation $\langle M\rangle_t=\sum_{i\in \tilde I_k}(-\log\bar F_\eps(\y_i\wedge t-g(i/n)))w(i/n)^2$: just note that the compensator of $\sum_{i\in \tilde I_k}{\bf 1}(\y_i\le t)$ equals the integrated hazard function $\sum_{i\in\tilde I_k}\int_0^{\y_i\wedge t-g(i/n)} \bar F_\eps(s)^{-1}dF_\eps(s)$. Moreover, $\tau:=\mY+Rh^\beta$ is a stopping time with respect to $({\cal F}_t)$. From the representation
\begin{align}\label{EqErrorMart}
\tilde\theta_k-\theta_k&=\frac{1}{n\lambda h}\Big(\sum_{i\in\tilde I_k}G_\eps(\y_i\wedge \tau -g(i/n))w(i/n)
-M(\tau)\Big)
\text{ with } G_\eps(z):=\lambda z+\log\bar F_\eps(z)
\end{align}
and the stopping Lemma \ref{LemStop}  in combination with $\E[\tau]<\infty$ due to the moment bound from Lemma \ref{LemTauMom} below we conclude
\[ \E[\tilde\theta_k-\theta_k]=\frac{1}{n\lambda h}\sum_{i\in\tilde I_k}\E[G_\eps(\y_i\wedge\tau-g(i/n))]w(i/n).
\]
In the case $\eps_i\sim\Exp(\lambda)$ we have $\log(\bar F_\eps(z))={-\lambda z}$, that is $G_\eps(z)=0$, and the estimator is unbiased.

By Assumption \eqref{EqPeps}, there is some $\delta>0$ such that $G_\eps(z)=O(z^2)$ holds for $z\in[0,\delta]$. For $z\ge \delta$ we have also $G_\eps(z)\le \lambda z=O(z^2)$ and $G_\eps(z)\ge -\abs{\log(\bar F_\eps(z))}$.
By Lemma \ref{LemTauMom}, Cauchy-Schwarz inequality, monotonicity of $\bar F_\eps$ and $\bar F_\eps(\eps_i)\sim U[0,1]$ this leads  to
\begin{align*}
&\abs{\E[\tilde\theta_k-\theta_k]}\\
&\lesssim\sum_{i\in\tilde I_k} \frac{\E[(\tau-g(i/n))^2]+\E[\log(\bar F_\eps(\y_i-g(i/n)))^2]^{1/2}P(\abs{\log(\bar F_\eps(\tau-g(i/n)))}>\delta)^{1/2}}{n h}\abs{w(i/n)}\\
&\lesssim\sum_{i\in\tilde I_k} \frac{\E[(\tau-g(i/n))^2]+\E[\log(\bar F_\eps(\eps_i))^2]^{1/2}P(\bar F_\eps(\min_{i\in \tilde I_k}\eps_i+2Rh^\beta)<e^{-\delta})^{1/2}}{n h}\abs{w(i/n)}\\
&\lesssim  \frac{(Rh^\beta+(n h)^{-1})^2+\bar F_\eps(\bar F_\eps^{-1}(e^{-\delta})-2Rh^\beta)^{nh/2}}{n h}\sum_{i\in\tilde I_k}\abs{w(i/n)}.
\end{align*}
With $h\to 0$ and $nh\to\infty$ the second term in the numerator converges geometrically fast to zero and the assertion for the bias of $\tilde\theta_n^{block}$ follows.

For the variance bound we use $\Var(A+B)\le 2\Var(A)+2\Var(B)$, $\Var(\sum_{i\in\tilde I_k}A_i)\le nh\sum_{i\in\tilde I_k}\E[A_i^2]$, $\abs{G_\eps(z)}\lesssim z+\abs{\log\bar F_\eps(z)}{\bf 1}(z>\delta)$ and the stopping Lemma \ref{LemStop} to obtain in analogy with the bias part:
\begin{align*}
&\Var(\tilde\theta_k)=\frac{1}{(n\lambda h)^2}\Var\Big(\sum_{i\in\tilde I_k}G_\eps(\y_i\wedge\tau-g(i/n))w(i/n)-M(\tau)\Big)\\
&\le \frac{2}{(n\lambda h)^2}\sum_{i\in \tilde I_k}\E\Big[\Big(nhG_\eps(\y_i\wedge\tau-g(i/n))^2+\abs{\log\bar F_\eps(\y_i\wedge\tau-g(i/n))}\Big)\Big]w(i/n)^2\\
&\lesssim \sum_{i\in \tilde I_k}\Big(\frac{\E\big[(\tau-g(i/n))^2\big]+\E\big[\abs{\log\bar F_\eps(\eps_i)}^4\big]^{1/2}P\big(\bar F_\eps\big(\min_{i\in \tilde I_k}\eps_i+2Rh^\beta\big)<e^{-\delta}\big)^{1/2}}{n h}\\
&\qquad + \frac{\E[\tau-g(i/n)]+\E\big[\abs{\log\bar F_\eps(\eps_i)}^2\big]^{1/2}P\big(\bar F_\eps\big(\min_{i\in \tilde I_k}\eps_i+2Rh^\beta\big)<e^{-\delta}\big)^{1/2}}{(n h)^2}\Big) w(i/n)^2\\
&\lesssim  \frac{(Rh^\beta+(n h)^{-1})^2+(Rh^\beta+(n h)^{-1})(nh)^{-1} + \bar F_\eps(\bar F_\eps^{-1}(e^{-\delta})-2Rh^\beta)^{nh/2}}{n h}\sum_{i\in \tilde I_k}w(i/n)^2\\
&\lesssim  \frac{(Rh^\beta+(n h)^{-1})^2} {n h}\sum_{i\in \tilde I_k}w(i/n)^2.
\end{align*}
For the global estimator we infer $\Var(\tilde\theta_n^{block})\lesssim h(Rh^\beta+(n h)^{-1})^2\norm{w}_2^2$ by independence of $(\tilde\theta_k)$. It remains to insert the rate-optimal choice of $h$ and to note that $n^{-4\beta/(\beta+1)}=o(n^{-(2\beta+1)/(\beta+1)})$ holds for $\beta>1/2$.

Finally, in the case $\eps_i\sim\Exp(\lambda)$ we have $\E[\mY-\max_ig(i/n)]\le (n\lambda h)^{-1}$ and $\Var(\tilde\theta_k)=\frac{\E[\langle M\rangle_\tau]}{(n\lambda h)^2}$. Consequently,
\[ \Var(\tilde\theta_k)\le\sum_{i\in \tilde I_k} \frac{\E[\mY+Rh^\beta -g(i/n)]}{\lambda(nh)^2}w(i/n)^2\le \frac{2Rh^\beta+(n\lambda h)^{-1}}{\lambda(nh)^2}\sum_{i\in \tilde I_k}w(i/n)^2,
\]
which by independence gives the asserted bound for $\Var(\tilde\theta_n^{block})$.
\end{proof}

\subsection{Adaptive estimation}

We now address the question of choosing the block size $h$ in a data-driven way, not assuming the regularity parameters $R$ and $\beta$ to be known. We apply Lepski's method \cite{Lepski} and treat the general regression-type model \eqref{EqRegr}. The main technical work is devoted to obtaining explicit critical values in Proposition \ref{Propcritval} of the appendix. To this end, the critical values are defined via the compensator of an exponential counting process and are thus itself again stochastic. While an explicit non-asymptotic risk analysis is clearly possible, we focus here  on the asymptotic risk, showing that by the versatility of Lepski's method  rate-optimal adaptive estimation up to logarithmic factors is indeed possible in our non-regular situation.

For a choice $n^{-1}(\log n)^2\le h_1<\ldots<h_M\le 1$ of block sizes $h_m$ with $h_m^{-1},nh_m\in\N$  consider the corresponding blockwise estimators
\[ \tilde\theta_{n,h_m}^{block}=\frac{1}{n}\sum_{k=0}^{h_m^{-1}-1}\sum_{i\in\tilde I_{k,h_m}}\Big(\y_i\wedge(\y_{k,h_m}^\ast+(nh_m)^{-1})-\lambda^{-1}{\bf 1}\Big(\y_i\le \y_{k,h_m}^\ast+(nh_m)^{-1}\Big)\Big)w(i/n),\]
where the subscript $h_m$ marks all quantities depending on the block size. Remark that the intercept $Rh_m^\beta$ in \eqref{EqTildethetak} has been replaced by the asymptotically balanced size $(nh_m)^{-1}$, which does not depend on the unknown $R$ and $\beta$. Among $(h_m)_{1\le m\le M}$ we select the block size adaptively via
\[ \hat h:=\inf\Big\{ h_m\,\Big|\, \exists m'\le m:\,\abs{\tilde\theta_{n,h_{m'}}^{block}-\tilde\theta_{n,h_{m+1}}^{block}}>\kappa_{m+1}+\kappa_{m'}\Big\}\wedge h_M
\]
with critical values ($k_i$ denotes the block $k$ with $i\in\tilde I_k$)
\[ \kappa_m=\sum_{i=1}^n\Big({\bf 1}(\y_i\le \y_{k_i,h_m}^\ast+(nh_m)^{-1})\frac{H_{\sqrt{c\log n}}(h_m^{1/2}w(i/n))}{n\lambda h_m^{1/2}}\Big)+\frac{(Cc\log n)^2\norm{w}_{1}}{(nh_m)^{2}\lambda}+\frac{\sqrt{c\log n}}{2n\lambda h_m^{1/2}}.
\]
Here, the function $H_x(y)=\frac{\log(1-2x\abs{y})}{-2x}-\abs{y}$  and the constant $C>0$ with property $\abs{\lambda z+\log(\bar F_\eps(z))}\le C^2 z^2$ for $z\in[0,\delta]$  are used and $c>0$ is specified below. Asymptotically, we have $H_x(y)\approx xy^2$ as $xy\to 0$ and $C\approx -(f_\eps'(0)+f_\eps(0)^2)$ in the case of a differentiable density $f_\eps$ of $\eps_i$ around zero (note $C=0$ for $\eps_i\sim\Exp(\lambda)$). Then the proof in the Appendix yields the following risk bounds.

\begin{theorem}\label{ThmLepski}
Assume  $g\in \cc^\beta(R)$, $\sup_x\abs{w(x)}<\infty$ and that $f_\eps/\bar F_\eps$ is bounded.
The adaptive estimator $\tilde\theta_{n}^{block}=\tilde\theta_{n,\hat h}^{block}$ satisfies with $h^\ast:=\sup\{ h_m\,|\, Rh_m^\beta\le (nh_m)^{-1}\}\vee h_1$
and for $\underline\lambda\in(0,\lambda)$, $n$ sufficiently large
\begin{align*}
\E[(\tilde\theta_{n}^{block}-\theta^{(n)})^2{\bf 1}(\hat h<h^\ast)] &\lesssim  M(n^{-c}+n^{(1-\underline\lambda c)/2}),\\ \E[(\tilde\theta_{n}^{block}-\tilde\theta_{n,h^\ast}^{block})^2{\bf 1}(\hat h\ge h^\ast)]&\lesssim \frac{(\log n)^4}{(nh^\ast)^4}+\frac{M\log n}{n^2 h^\ast}.
\end{align*}
Choosing $c>5\lambda^{-1}\vee 2$ and asymptotically $h_0\thicksim (\log n)^2n^{-1}$, $h_m\thicksim h_0q^m$ for $m=1,\ldots,M$, $q>1$, $M\thicksim \log n$, the adaptive estimator exhibits the asymptotic rate
\[ \E[(\tilde\theta_{n}^{block}-\theta^{(n)})^2] \lesssim (\log n)^2n^{-(2\beta+1)/(\beta+1)}+(\log n)^4n^{-4\beta/(\beta+1)}.\]
In particular, for $\beta\ge 1/2$ the estimator achieves the minimax optimal rate up to a logarithmic factor.
\end{theorem}

\begin{remark}
The oracle-type block size $h^\ast$ is the largest block size among $(h_m)$ such that $\tilde\theta_{n,h_m}^{block}$ remains unbiased, except for the distribution bias induced in the case $\eps_i\not\sim\Exp(\lambda)$.
As the proof shows, in the case $\eps_i\sim\Exp(\lambda)$ not only the critical values ($C=0$), but also the bounds simplify. We obtain $\E[(\tilde\theta_{n}^{block}-\tilde\theta_{n,h^\ast}^{block})^2{\bf 1}(\hat h>h^\ast)]\lesssim \frac{M\log n}{n^2 h^\ast}$ and thus the minimax optimal rate up to a log-factor for any $\beta>0$.

More elaborate arguments in the proof could give a smaller exponent for the logarithmic factor, but it is quite likely that some logarithmic factor has  to be paid necessarily for adaptation, cf. \cit{jirakmeisterreiss} for  a related result. Similarly, the hypotheses that $w$ and $f_\eps/\bar F_\eps$ are uniformly bounded are certainly not necessary, but permit  more concise and transparent bounds.
\end{remark}

\subsection{Rate optimality}

We prove  that the rate $R^{1/(2\beta+2)}n^{-(\beta+1/2)/(\beta+1)}$ is optimal in a minimax sense over $\cc^\beta(R)$. The proof is conducted for the PPP model, the regression case with $\eps_i\sim\Exp(\lambda)$ can be treated in the same way.

\begin{theorem}\label{ThmLower}
For estimating $\theta=\int_0^1g(x)w(x)\,dx$, $w\in L^2([0,1])$,  over the parameter class $\cc^\beta(R)$, $\beta\in(0,1]$, $R>0$, the following asymptotic lower bound holds:
\[ \liminf_{n\to\infty}\inf_{\hat\theta_n}\sup_{g\in\cc^\beta(R)}R^{-1/(\beta+1)}n^{(2\beta+1)/(\beta+1)}\norm{w}_{L^2}^{-2}\E_g[(\hat\theta_n-\theta)^2]>0.\]
The infimum extends over all estimators $\hat\theta_n$ from the PPP model with intensity \eqref{EqPPP}.
\end{theorem}

\begin{proof}
The proof is based on a Bayesian risk bound, which clearly provides a lower bound for the minimax risk, see \cit{korost:tsyb:book} for similar approaches.
Take an independent Bernoulli sequence $(\eps_k)$, i.e. $P(\eps_k=1)=p$, $P(\eps_k=0)=1-p$ with $p\in(0,1)$, and set for a triangular kernel $K(y)=2\min(y,1-y){\bf  1}_{[0,1]}(y)$
\[ g(x)=\sum_{k=0}^{h^{-1}-1}\eps_kg_k(x)\text{ with }g_k(x)=cR h^\beta K((x-kh)/h),\]
where $h\in(0,1)$ with $h^{-1}\in\N$ will be chosen later. Then for $c>0$ sufficiently small, we have $g\in C^\beta(R)$ for all $h$ and all realisations of $(\eps_k)$. We interpret this specification of $g$ as a prior on $\cc^\beta(R)$ and we shall make use of the independence of prior as well as the observation laws on different blocks $I_k=[kh,(k+1)h)$. For each $k$ we obtain from the Bayes formula the posterior probability given the observations of the PPP in interval $I_k$ (cf. the likelihood derivation in \eqref{EqLik} below)
\[ \hat\eps_k:=P(\eps_k=1\,|\,(X_i,Y_i)_{i\ge 1})=\frac{p e^{n\int g_k}{\bf 1}(\forall X_i\in I_k: Y_i\ge g_k(X_i))}{1-p+pe^{n\int g_k}}.
\]
Using that $\eps_k$ are 0-1-valued, we have $\hat\eps_k=\E[\eps_k\,|\,(X_i,Y_i)_{i\ge 1}]$ and
$\Var(\eps_k\,|\,(X_i,Y_i)_{i\ge 1})=\hat\eps_k(1-\hat\eps_k)$. Therefore
 the Bayes-optimal estimator of $\theta$ under squared loss is given by the posterior mean
\[ \hat\theta=\sum_{k=0}^{h^{-1}-1}\hat\eps_k\int_{I_k}g_k(x)w(x)\,dx.\]
Using independence and $\E[\hat\eps_k-\eps_k]=0$, its Bayes risk is calculated as
\begin{align*}
\E[(\hat\theta-\theta)^2] &=\sum_{k=0}^{h^{-1}-1}\Var(\hat\eps_k-\eps_k)\Big(\int_{I_k}g_k(x)w(x)\,dx\Big)^2\\
&= \sum_{k=0}^{h^{-1}-1}\E\Big[\Var\Big(\eps_k\,\Big|\,(X_i,Y_i)_{i\ge 1}\Big)\Big]\Big(\int_{I_k}g_k(x)w(x)\,dx\Big)^2\\
&= \sum_{k=0}^{h^{-1}-1}\frac{p(1-p)}{(1-p+pe^{n\int g_k})^2}\Big(\int_{I_k}g_k(x)w(x)\,dx\Big)^2.
\end{align*}
We now choose $h=\ceil{(cRn)^{1/(\beta+1)}}^{-1}$ such that $n\int_{I_k}g_k(x)\,dx=cR h^{\beta+1} n\le 1$
holds. Then the Bayes risk is bounded in order by
\[ \E[(\hat\theta-\theta)^2]\gtrsim R^2h^{2\beta+1}\sum_{k=0}^{h^{-1}-1}\Big(\int_{I_k}\frac{K((x-kh)/h)}{\norm{K((\cdot-kh)/h)}_{L^2}}w(x)\,dx\Big)^2.\]
The same argument over the shifted blocks $I_k'=[(k+1/2)h,(k+3/2)h)$ implies that the minimax risk is bounded by the maximum (and thus the average) over the respective Bayes risks:
\[ \inf_{\hat\theta_n}\sup_{g\in\cc^\beta(R)}\E_g[(\hat\theta_n-\theta)^2]\gtrsim R^2h^{2\beta+1}\sum_{k=0}^{2h^{-1}-2}\Big(\int_{I_k}\frac{K((x-kh/2)/h)}{\norm{K((\cdot-kh/2)/h)}_{L^2}}w(x)\,dx\Big)^2.\]

The tent functions $K((x-kh/2)/h)$, $k=0,1,\ldots,2h^{-1}-2$ form a Riesz basis for their linear span $V_h$ (see e.g. Example 2.1 in \cit{Wojtaszczyk}), which is the space of all linear splines with knots at $kh/2$, vanishing at the boundary. This means that  the sum in the last display is larger than a constant times the $L^2([0,1])$-norm of the orthogonal projection of $w$ onto $V_h$. As $\bigcup_{h>0}V_h$ is dense in $L^2([0,1])$, the $L^2$-norm of the projections of $w$ onto $V_h$ converges for $h\to 0$ to the $L^2$-norm of $w$. Insertion of $h\thicksim (Rn)^{-1/(\beta+1)}$ yields the desired lower bound rate $R^{1/(\beta+1)}n^{-(2\beta+1)/(\beta+1)}\norm{w}_{L^2}^2$.
\end{proof}

\section{Nonparametric Maximum-Likelihood}

\subsection{The MLE over $\cc^{\beta}(R)$}

Let us study the nonparametric maximum-likelihood estimator (MLE) in the class $\cc^{\beta}(R)$. Denote by $P_g$ the law of the observations in the PPP model with intensity $\lambda_g(x,y)=n{\bf 1}(y\ge g(x))$. Then  for $g\ge g_0$ by  Thm. 1.3 in \cit{Kutoyants} and the fact that the PPP intensities coincide outside the compact set $[0,1]\times[\min g_0,\max g]$ we obtain  the Radon-Nikodym-derivative
\[ \frac{dP_g}{dP_{g_0}}=\exp\Big(n\int_0^1(g-g_0)(x)\,dx\Big){\bf 1}\Big(\forall i:Y_i\ge g(X_i)\Big).\]
A simple probability measure $P_0$ dominating all $P_g$, $g\in \cc^\beta(R)$ (where $g$ need not be bounded from below), is given by the PPP model with intensity $\lambda_0(x,y)=n(e^y\wedge 1)$ and yields again via  Thm. 1.3 in \cit{Kutoyants} the likelihood
\begin{align}
&{\cal L}(g) = \frac{dP_g}{dP_0}\nonumber\\
&= \Big(\prod_{j\ge 1}\frac{n{\bf 1}(Y_j\ge g(X_j))}{n(e^{Y_j}\wedge 1)}\Big)\exp\Big(-n\int_0^1\int_{-\infty}^\infty ({\bf 1}(y\ge g(x))-e^y\wedge 1)\,dy\,dx\Big)\nonumber\\
&=\Big(\prod_{j\ge 1} e^{(-Y_j)_+}{\bf 1}(Y_j\ge g(X_j))\Big)\exp\Big(-n\int_0^1(-1-g(x))\,dx\Big)\nonumber\\
&=\exp\Big(n+\sum_{j\ge 1} (-Y_j)_+\Big)\exp\Big(n\int_0^1g(x)\,dx\Big){\bf 1}\Big(\forall j\ge 1:\:Y_j\ge g(X_j)\Big).\label{EqLik}
\end{align}
The first factor is independent of $g$ and we obtain thus the same structure as under $P_{g_0}$ above.
The MLE over $\cc^{\beta}(R)$ is the function $\hat g$ that maximizes $\int_0^1g$ over all $g\in \cc^{\beta}(R)$ with $g(X_j)\le Y_j$ for all $j$. We can write explicitly
\[ \hat g^{MLE}(x)=\min_{j\ge 1}\Big(Y_j+R\abs{x-X_j}^\beta\Big),\]
since the right-hand side even maximises $g(x)$ pointwise over the considered class of $g$, see also Figure 2.
The corresponding likelihood (with respect to $n$-dimensional Lebesgue measure) in the regression-type model \eqref{EqRegr} with $\eps_i\sim \Exp(\lambda)$ i.i.d. is given by
\[ {\cal L}^{regr}(g)
=\lambda^n \exp\Big(-\lambda\sum_{i=1}^n \y_i\Big)\exp\Big(\lambda\sum_{i=1}^n g(i/n)\Big){\bf 1}\Big(\forall i=1,\ldots,n:\;\y_i\ge g(i/n)\Big).
\]
The maximum-likelihood estimator over $\cc^\beta(R)$ is then similarly given by
\begin{equation}\label{EqMLERegr}
 \hat g^{MLE-regr}(x)=\min_{i=1,\ldots,n}\Big(\y_i+R\abs{x-i/n}^\beta\Big), \quad x\in[0,1],
 \end{equation}
see Figure \ref{fig-est-mle} for an illustration of the two constructions of the MLE. They are both quickly determined numerically. In the sequel, we shall focus on the MLE in the PPP model and only hint at the parallel theory for the regression-type model under exponential noise. We abstain from analysing the exponential MLE under non-exponential noise in the regression-type model because the results must be asymptotic in nature and will be comparable to Theorem \ref{ThmRegr}.

\begin{figure}[t]\begin{center}
\begin{minipage}[t]{0.49\textwidth}
\includegraphics[width=\textwidth]{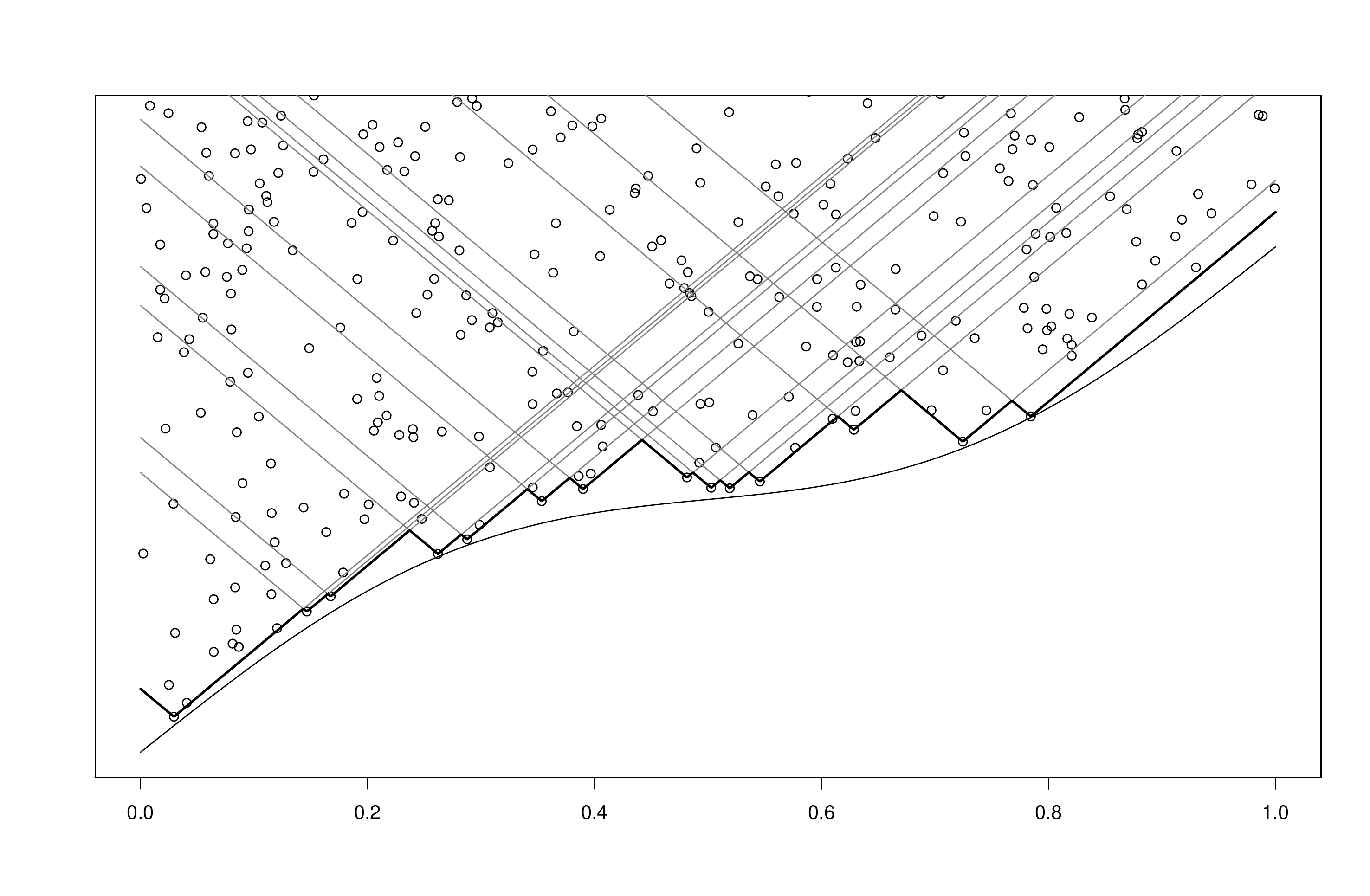}
\end{minipage}
\begin{minipage}[t]{0.49\textwidth}
\includegraphics[width=\textwidth]{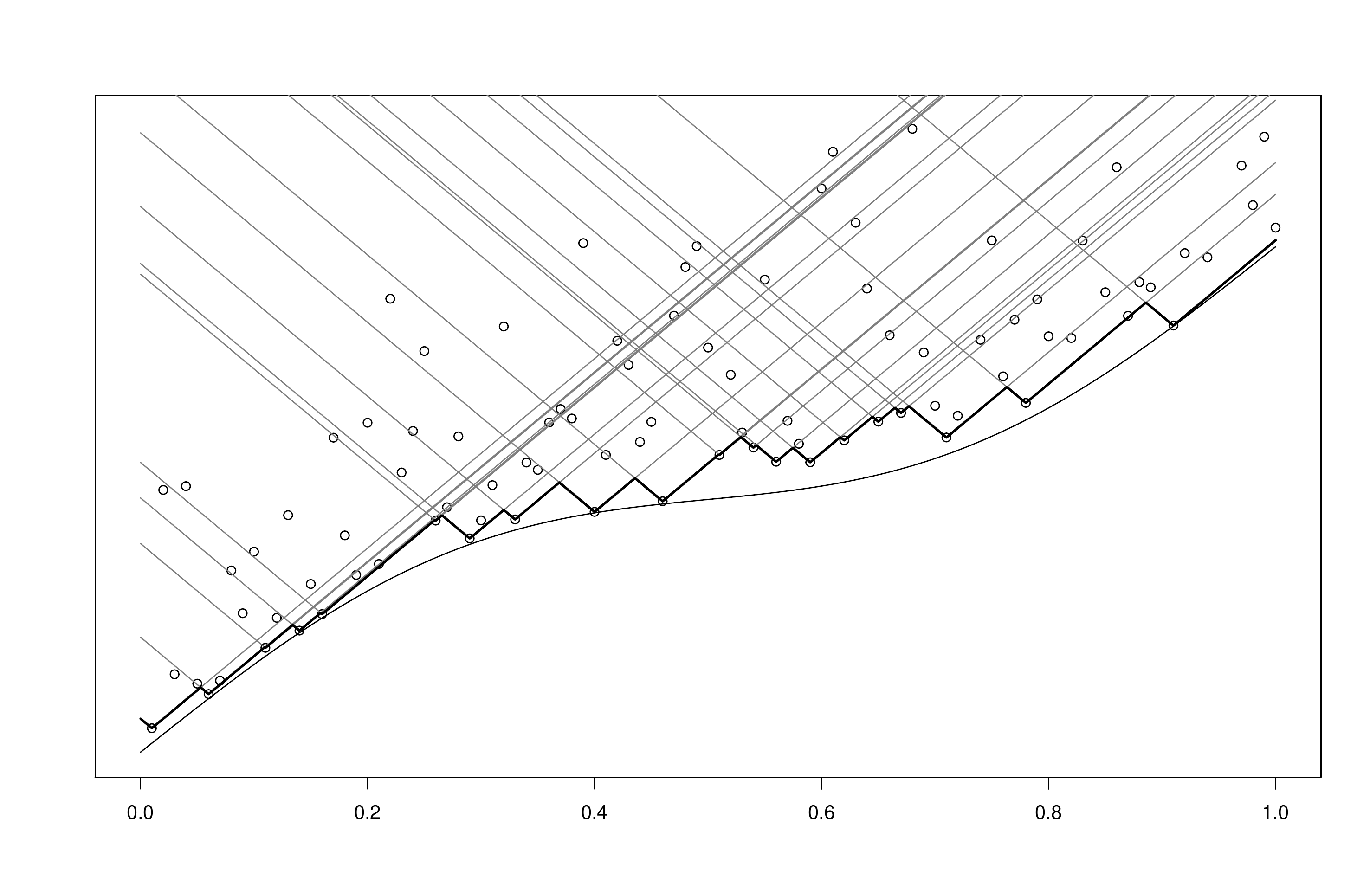}
\end{minipage}\end{center}\vspace{-10mm}

\caption{Construction of $\hat{g}^{MLE}$ in the PPP model ($n=100$, left) and in the regression-type model ($n=100$, $\eps_i\sim\Exp(1)$, right) for $\beta=1$. Thin lines indicate $x\mapsto Y_j+R\abs{x-X_j}$ at observations on the graph of $\hat{g}^{MLE}$.}\label{fig-est-mle}\end{figure}

\begin{proposition}\label{PropSuff1}
The nonparametric MLE $(\hat g^{MLE}(x),x\in[0,1])$ is a sufficient and complete statistics for $\cc^\beta(R)$.
\end{proposition}

\begin{proof}
By definition of $\hat g^{MLE}$ the likelihood \eqref{EqLik} can be written as
\[{\cal L}(g)=\exp\Big(n+\sum_{j\ge 1} (-Y_j)_+\Big)\exp\Big(n\int_0^1g(x)\,dx\Big){\bf 1}\Big(g\le\hat g^{MLE}\Big)\]
such that by Neyman's factorisation criterion (e.g. \cit{lehmann}) $\hat g^{MLE}$ is a sufficient statistics for this parameter class.

Let us  remark that by definition $\hat g^{MLE}$ is an element of $\cc^\beta(R)$. Since $\cc^\beta(R)$, equipped with its $\cc^\beta$-norm, is not separable, we equip it with the Borel $\sigma$-algebra generated by the uniform (supremum) norm, which is generated by all point evaluations.
Measurability of the estimator $\hat g^{MLE}$ is then easily established since all point evaluations $\hat g^{MLE}(x)$, $x\in[0,1]$, are measurable as minima of countably many random variables.

For completeness we now consider any statistic $T:\cc^\beta(R)\to\R$ satisfying $\E_g[T(\hat g^{MLE})]=0$ for all $g\in\cc^\beta(R)$, which is Borel measurable with respect to the uniform norm. For $g\in \cc^\beta(R)$ denote by $[g,\infty):=\{ h\in \cc^\beta(R)\,|\, h\ge g\}$ the 'bracket' between $g$ and $\infty$, that is all functions whose graphs lie above $g$. Noting $[g,\infty)\cap [h,\infty)=[g\vee h,\infty)$ where the maximum $g\vee h$ is again in $\cc^\beta(R)$, the family $\{[g,\infty)\,|\,g\in \cc^\beta(R)\}$ is an $\cap$-stable generator of the uniform Borel $\sigma$-algebra in $\cc^\beta(R)$:
for any $x_0\in[0,1]$, $y_0\in\R$ we have $\{h\in\cc^\beta(R)\,|\,h(x_0)\ge y_0\}=[y_0-R\abs{\cdot-x_0}^\beta,\infty)$ by the H\"older condition and $\{[y_0,\infty)\,|\,y_0\in\R\}$ generates the Borel $\sigma$-algebra on $\R$.

From $\E_{g}[T(\hat g^{MLE})]=0$ we obtain by using the likelihood under $P_0$
\[ e^{n\int (g+1)}\E_{0}\Big[T(\hat g^{MLE})e^{\sum_{j\ge 1} (-Y_j)_+}{\bf 1}(\hat g^{MLE}\in[g,\infty))\Big]=0.\]
Splitting $T=T^+-T^-$ with non-negative $T^+,T^-$, we infer that the measures $B\mapsto \E_{0}[T^\pm(\hat g^{MLE})e^{\sum_{j\ge 1} (-Y_j)_+}{\bf 1}(\hat g^{MLE}\in B)]$ agree on $\{[g,\infty)\,|\,g\in \cc^\beta(R)\}$ and thus by the uniqueness theorem for all uniform Borel sets $B$ in $\cc^\beta(R)$, in particular for $B=\{T>0\}$ and $B=\{T<0\}$. This implies $T^+(\hat g^{MLE})e^{\sum_{j\ge 1} (-Y_j)_+}=T^-(\hat g^{MLE})e^{\sum_{j\ge 1} (-Y_j)_+}$ $P_{0}$-a.s. and thus $T(\hat g^{MLE})=0$ $P_g$-a.s. for all $g\in \cc^\beta(R)$.
\end{proof}

In analogy with the block-wise estimator $\hte^{block}$ we set
\[ \hte^{MLE}:=\int_0^1 \hat g^{MLE}(x)w(x)\,dx-\frac1n\sum_{j\ge 1}{\bf 1}\big(\hat g^{MLE}(X_j)=Y_j\big)w(X_j) .\]
This means that $\hte^{MLE}$ is obtained by a plug-in of the nonparametric MLE $\hat g^{MLE}$ into the functional minus a bias correction which counts the relative number of observations on the graph of $\hat g^{MLE}$. The striking result is that this estimator is not only unbiased, but even uniformly of minimum variance among all unbiased estimators for the class $\cc^{\beta}(R)$ (UMVU).

\begin{theorem}\label{TheoMLE1}
The estimator $\hte^{MLE}$ is for each finite sample size $n$ UMVU over the class $\cc^{\beta}(R)$ with
\begin{align*}
&\Var(\hte^{MLE})=\frac1{n}\int_0^1\E[\hat g^{MLE}(x)-g(x)]w(x)^2\,dx\\
&\le \Big(\Gamma(\beta/(\beta+1))\beta(2R/(\beta+1))^{1/(\beta+1)}n^{-(2\beta+1)/(\beta+1)}+\frac1{n^2} e^{-2\beta Rn/(\beta+1)}\Big)\norm{w}_{L^2}^2.
\end{align*}
For $n\to\infty$ we obtain
\[ \Var(\hte^{MLE})\le (2+o(1)) R^{1/(\beta+1)}n^{-(2\beta+1)/(\beta+1)}\norm{w}_{L^2}^2.\]
\end{theorem}

\begin{proof}
Let us define another weighted counting process
\[\bar N(t)=\sum_{j\ge 1}{\bf 1}\Big(Y_j\le t\wedge \min_{i\ge 1}(Y_i+R\abs{X_j-X_i}^\beta)\Big)w(X_j), \quad t\in\R.\]
Note  $\{\min_{i\ge 1}(Y_i+R\abs{X_j-X_i}^\beta)<t\}=\{\min_{i:Y_i<t}(Y_i+R\abs{X_j-X_i}^\beta)<t\}$ for each $t$ and that the pure ($w=1$) counting process has  stochastic intensity
\[\bar\lambda_t=n\int_0^1\int_{[g(x),t]}{\bf 1}\Big(\min_{i:Y_i< s}(Y_i+R\abs{x-X_i}^\beta)\ge s\Big)\,dsdx.\]
Consequently, $\bar N$ is  adapted to $({\cal F}_t)$ from \eqref{EqFt} and we obtain  by compensation (cf. Prop.  2.32 in \cit{karr}) the $({\cal F}_t)$-martingale
\[ \bar M(t)=\bar N(t)-n\int_0^1\int_{[g(x),t]}{\bf 1}\Big(\min_{i:Y_i\le s}(Y_i+R\abs{x-X_i}^\beta)\ge s\Big)\,ds\,w(x)\,dx.\]

The main observation is the identity
\begin{align*}
\lim_{t\to\infty} (\bar N(t)-\bar M(t)) &= n\int_0^1\int_{g(x)}^\infty {\bf 1}\Big(\min_{i\ge 1}(Y_i+R\abs{x-X_i}^\beta)\ge s\Big)\,ds\,w(x)\,dx\\
 &= n\int_0^1\int_{g(x)}^\infty {\bf 1}(\hat g^{MLE}(x)\ge s)\,ds\,w(x)\,dx\\
&=n\int_0^1 (\hat g^{MLE}(x)-g(x))w(x)\,dx,
\end{align*}
which tells us that
\[\bar N(\infty):=\lim_{t\to\infty} \bar N(t)=\sum_{j\ge 1}{\bf 1}(\hat g^{MLE}(X_j)\ge Y_j)w(X_j)=\sum_{j\ge 1}{\bf 1}(\hat g^{MLE}(X_j)=Y_j)w(X_j)\]
simultaneously counts the weighted number of points $(X_j,Y_j)$ on the graph of $\hat g^{MLE}$ and
 equals the scaled bias $n(\int_0^1 \hat g^{MLE}(x)w(x)dx-\theta)$ up to a martingale term. We thus have
$\hte^{MLE}=\int_0^1 \hat g^{MLE}(x)w(x)\,dx-\frac1n \bar N(\infty)=\theta-\frac1n \bar M(\infty)$
where
\[ \bar M(\infty)=\sum_j{\bf 1}\Big(Y_j\le\hat g^{MLE}(x)\Big)w(X_j)-\int_0^1 (\hat g^{MLE}(x)-g(x))w(x)\,dx\]
 is the a.s. and $L^2$-limit of the $L^2$-bounded martingale $\bar M$ with
\begin{align*}
\langle \bar M\rangle_t &= n\int_0^1\int_{[g(x),t]}{\bf 1}\Big(\min_{i:Y_i< s}(Y_i+R\abs{x-X_i}^\beta)\ge s\Big)\,ds\,w(x)^2dx\\
& \uparrow n
\int_0^1(\hat g^{MLE}(x)-g(x))w(x)^2\,dx=:\langle\bar M\rangle_\infty\text{ as $t\uparrow \infty$.}
\end{align*}
We obtain from $\E[\hte^{MLE}-\theta]=\frac1n\E[-\bar M(\infty)]$, $\Var(\hte^{MLE})=\frac1{n^2}\Var(\bar M(\infty))$ the result (use Lemma \ref{LemStop} with $\tau=\infty$)
\[ \E[\hte^{MLE}]=\theta\text{ and } \Var(\hte^{MLE})=\frac1{n^2}E[\langle\bar M\rangle_\infty]=\frac1{n}\int_0^1\E[\hat g^{MLE}(x)-g(x)]w(x)^2\,dx.
\]
Hence, $\hte^{MLE}$ is an unbiased estimator and by the Lehmann-Scheff\'e Theorem $\hte^{MLE}$, derived from a sufficient and complete statistics, is uniformly of minimum variance among all unbiased estimators (e.g. \cit{lehmann}).

To bound the variance we use a universal, but somewhat rough deviation bound for $s\ge 0$ and $x\in[0,1]$:
\begin{align}
P(\hat g^{MLE}(x)-g(x)\ge s)&=\exp\Big(-n\int_0^1(s-R\abs{\xi-x}^\beta+g(x)-g(\xi))_+d\xi\Big)\nonumber\\
&\le\exp\Big(-n\int_0^1(s-2R\abs{\xi-x}^\beta)_+d\xi\Big)\nonumber\\
&\le\begin{cases}\exp(-n\frac{2R}{\beta+1}(s/2R)^{(\beta+1)/\beta}),& s\in[0,2R],\\ \exp(-n(s-2R/(\beta+1))),& s>2R.\end{cases}\label{EqgMLE}
\end{align}
In the first step we have evaluated the probability that no observation lies in $\{(\xi,\eta)\,|\,\eta+R\abs{x-\xi}^\beta<g(x)+s\}$ using the PPP property. Integrating this survival function bound, we obtain directly
\begin{align}
&\E[\hat g^{MLE}(x)-g(x)]= \int_0^\infty P(\hat g^{MLE}(x)-g(x)\ge s)\,ds\nonumber\\
&\quad\le\int_{0}^{2R} \exp\Big(-n\frac{2R}{\beta+1}(s/2R)^{(\beta+1)/\beta}\Big)ds+\int_{2R}^\infty e^{-n(s-2R/(\beta+1))}ds\nonumber\\
&\quad = 
\Gamma(\beta/(\beta+1))\beta(2R/(\beta+1))^{1/(\beta+1)}n^{-\beta/(\beta+1)}+\frac1n e^{-2\beta Rn/(\beta+1)}.\label{EqExpBound}
\end{align}
Insertion and a numerical evaluation then yield (the maximal constant being attained for $\beta\to 0$)
$\Var(\hte^{MLE})\le (2+o(1)) R^{1/(\beta+1)}\norm{w}_{L^2}^2n^{-(2\beta+1)/(\beta+1)}$.
\end{proof}

\begin{remark}
The MLE $\hat g^{MLE-regr}$ from \eqref{EqMLERegr} for the regression-type model is by the same (or simpler) arguments a sufficient and complete statistics over $\cc^\beta(R)$. It gives rise to the estimator
\[\hat\theta^{MLE-regr}=\frac1n\sum_{i=1}^n \Big(\hat g^{MLE-regr}(i/n)-\lambda^{-1} {\bf 1}\Big(\hat g^{MLE-regr}(i/n)=\y_i\Big)\Big)w(i/n).\]
Then for $\Exp(\lambda)$-distributed errors $\hat\theta^{MLE-regr}$ is an unbiased estimator of $\theta^{(n)}$ with
$\Var(\hat\theta^{MLE-regr})=\frac{1}{n^2\lambda}\sum_{i=1}^n\E[\hat g^{MLE-regr}(i/n)-g(i/n)]w(i/n)^2$. This follows analogously from the corresponding counting process $\bar N(t)$, replacing $X_j$ in the PPP case by $j/n$.
The asymptotic upper bound for the regression model as $n\to\infty$ is the same as for the PPP model, but with the noise level $1/n$ replaced by $1/(n\lambda)$, provided $w^2$ is Riemann-integrable.
\end{remark}

While $\hat\theta^{MLE}$ as an UMVU estimator enjoys very desirable finite sample properties of its risk, for inference questions we are also in need of distributional properties, at least asymptotically. A priori, in our Poisson-type boundary models it might not be clear whether the limiting distribution is Gaussian, but in fact this is the case since we average  over the interval $[0,1]$.
The proof of the following central limit theorems for the Lipschitz case is slightly more technical and therefore given in the appendix. Note that a  central limit theorem for the blockwise estimator $\hat\theta^{block}_n$ follows far more easily due to Lindeberg's theorem, profiting from the independence between blocks.

\begin{theorem}\label{ThmCLT1}
For $g\in \cc^1(R)$ (Lipschitz case), $\sup_{x\in[0,1]}\abs{w(x)}<\infty$  and  $\Var(\hat\theta_n^{MLE})\thicksim n^{-3/2}$, indicating the dependence of $\hat\theta^{MLE}$ on $n$, the following central limit theorems hold as $n\to\infty$:
\begin{align*}
\frac{n^{1/2}(\hte_n^{MLE}-\theta)}{(\int_0^1\E[\hat g^{MLE}(x)-g(x)]w(x)^2\,dx)^{1/2}}&\Rightarrow N(0,1),\\
\frac{n^{1/2}(\hte_n^{MLE}-\theta)}{(\int_0^1(\hat g^{MLE}(x)-g(x))w(x)^2\,dx)^{1/2}}&\Rightarrow N(0,1).
\end{align*}
Furthermore,   the following self-normalising version is valid:
\[ \frac{n^{1/2}(\hte_n^{MLE}-\theta)}{(\frac1n\sum_{j\ge 1}{\bf 1}(\hat g^{MLE}(X_j)=Y_j)w(X_j)^2)^{1/2}}\Rightarrow N(0,1).\]
\end{theorem}

\begin{remark}
The 'super-efficient' case $\Var(\hat\theta_n^{MLE})=o(n^{-3/2})$ is to some extent degenerate and might possibly result in non-Gaussian limit laws. A lower estimate in \eqref{EqgMLE} above shows that $\E[\hat g^{MLE}(x)-g(x)]\gtrsim n^{-3/2}$ holds as soon as the function $g$ satisfies $\abs{g(y)-g(x)}\le R'\abs{y-x}$ for $R'<R$ and $y$ in a neighbourhood of $x$. This means that $\Var(\hat\theta^{MLE})\thicksim n^{-3/2}$ and the CLTs above are applicable whenever $g$ has a local Lipschitz constant  smaller than $R$, at least on some subinterval. Note that in this case we also get 'for free' the nice geometric result that the number of observations  on the graph of $\hat g^{MLE}$ is of order $n^{1/2}$ (in mean) because of
\[\frac1n\sum_{j\ge 1} {\bf 1}(\hat g^{MLE}(X_j)=Y_j)\thicksim \int_0^1 \E[\hat g^{MLE}(x)-g(x)]\,dx\thicksim n^{-1/2}.\]
The standard deviation is of smaller order as the proof of Theorem \ref{ThmCLT1} shows.
\end{remark}

An immediate consequence of the selfnormalising CLT is the following inference statement.

\begin{corollary}\label{CorCI}
Under the assumptions of Theorem \ref{ThmCLT1}
\[ {\cal I}_n:=\Big[\hat\theta_n^{MLE}-\hat\sigma_nq_{1-\alpha/2},\hat\theta_n^{MLE}+\hat\sigma_nq_{1-\alpha/2}\Big],\, \hat\sigma_n^2:=\frac{1}{n^2}\sum_{j\ge 1}{\bf 1}(\hat g^{MLE}(X_j)=Y_j)w(X_j)^2,
\]
with $q_{1-\alpha/2}$ the $(1-\alpha/2)$-quantile of $N(0,1)$, is a confidence interval for $\theta$  with asymptotic coverage $1-\alpha$.
\end{corollary}

Also the asymptotic variance can be determined explicitly.

\begin{corollary}\label{CorAVAR}
Under the assumptions of Theorem \ref{ThmCLT1} we obtain
\[ \Var(\hat\theta_n^{MLE})=\Big(\frac{\sqrt{\pi}}{2}+o(1)\Big)n^{-3/2}\int_0^1\sqrt{(R^2-g'(x)^2)/R}w(x)^2dx,\]
where $g'$ denotes the weak derivative of the Lipschitz function $g$, and thus
\[ n^{3/4}(\hte_n^{MLE}-\theta)\Rightarrow N\Big(0,\int_0^1\sqrt{(R^2-g'(x)^2)/R}w(x)^2dx\Big).\]
\end{corollary}

\begin{proof}
A Lipschitz function $g$ is absolutely continuous, hence a.e. differentiable and necessarily $\abs{g'(x)}\le R$ holds a.e.
For $x\in(0,1)$ where $g'(x)$ exists we obtain, arguing by dominated convergence using \eqref{EqgMLE},
\begin{align*}
&P\Big(n^{1/2}(\hat g_n^{MLE}(x)-g(x))\ge z\Big)\\
&=\exp\Big(-n\int_0^1 (g(x)+zn^{-1/2}-R\abs{\xi-x}-g(\xi))_+ d\xi\Big)\\
 &=\exp\Big(-\int_{-n^{1/2}x}^{n^{1/2}(1-x)} \big(n^{1/2}(g(x)-g(x+n^{-1/2}u))+z-R\abs{u}\big)_+ du\Big)\\
&\rightarrow \exp\Big(-\int_{-\infty}^\infty (z-R\abs{u}-g'(x)u)_+du\Big)= \exp\Big(-\frac{R}{R^2-g'(x)^2}z^2\Big).
\end{align*}
By integrating this survival function over $z\in\R^+$ and applying dominated convergence due to the uniform bound \eqref{EqgMLE}, we conclude
\[ n^{1/2}\E[\hat g_n^{(MLE)}(x)-g(x)]\rightarrow \sqrt{(R^2-g'(x)^2)/R}\frac{\sqrt{\pi}}{2}.\]
Integration over $x$ yields by another application of dominated convergence in view of \eqref{EqExpBound} the asymptotic expression for $\Var(\hat\theta_n^{MLE})$.
\end{proof}

The last corollary shows that for constant $g$ the  asymptotic variance (rescaled by $n^{3/2}$) equals $\sqrt{R\pi}\norm{w}_{L^2}^2/2$ and is largest among all admissible $g$ while for linear $g$ with slope $\pm R$ the rescaled asymptotic variance vanishes, i.e. the convergence rate is faster than $n^{-3/2}$. In Figure 2 we see indeed that $\hat g^{MLE}$ is closest to $g$ where $g$ has largest slope. Notice that the bias correction via point counts gives a precise meaning for this observation. So far, our methods of proof do not extend to the $\beta$-H\"older case with $\beta<1$ or to $w\in L^2$ because we need to control the difference to a block-wise partitioned MLE. The strategy of proof does neither apply to the monotone MLE, as introduced next.

\subsection{MLE under monotonicity}

Let us consider the general nonparametric class
\[{\cal M}:=\{g:[0,1)\to\R\,|\, g \text{ is increasing and left-continuous}\}\]
of monotone, that is (not necessarily strictly) increasing functions. Since monotone $g$ have at most countably many jumps, the observations for left- and right-continuous versions of $g$ are a.s. identical. Then the nonparametric MLE for the PPP model over this class is given by
\[ \hat g^{Mon}(x)=\min_{i:X_i\ge x}Y_i,\quad x\in[0,1),\]
which is obvious from the fact that any $g\in {\cal M}$ with $g(X_i)\le Y_i$ for all $i$ necessarily satisfies $g\le \hat g^{Mon}$, see also Figure \ref{fig-est-mon}. Note that a.s. $\hat g^{Mon}(x)<\infty$ holds for $x\in[0,1)$, but $\lim_{x\uparrow 1}\hat g^{Mon}(x)=\infty$.

\begin{figure}[t]\begin{center}
\begin{minipage}[t]{0.49\textwidth}
\includegraphics[width=\textwidth]{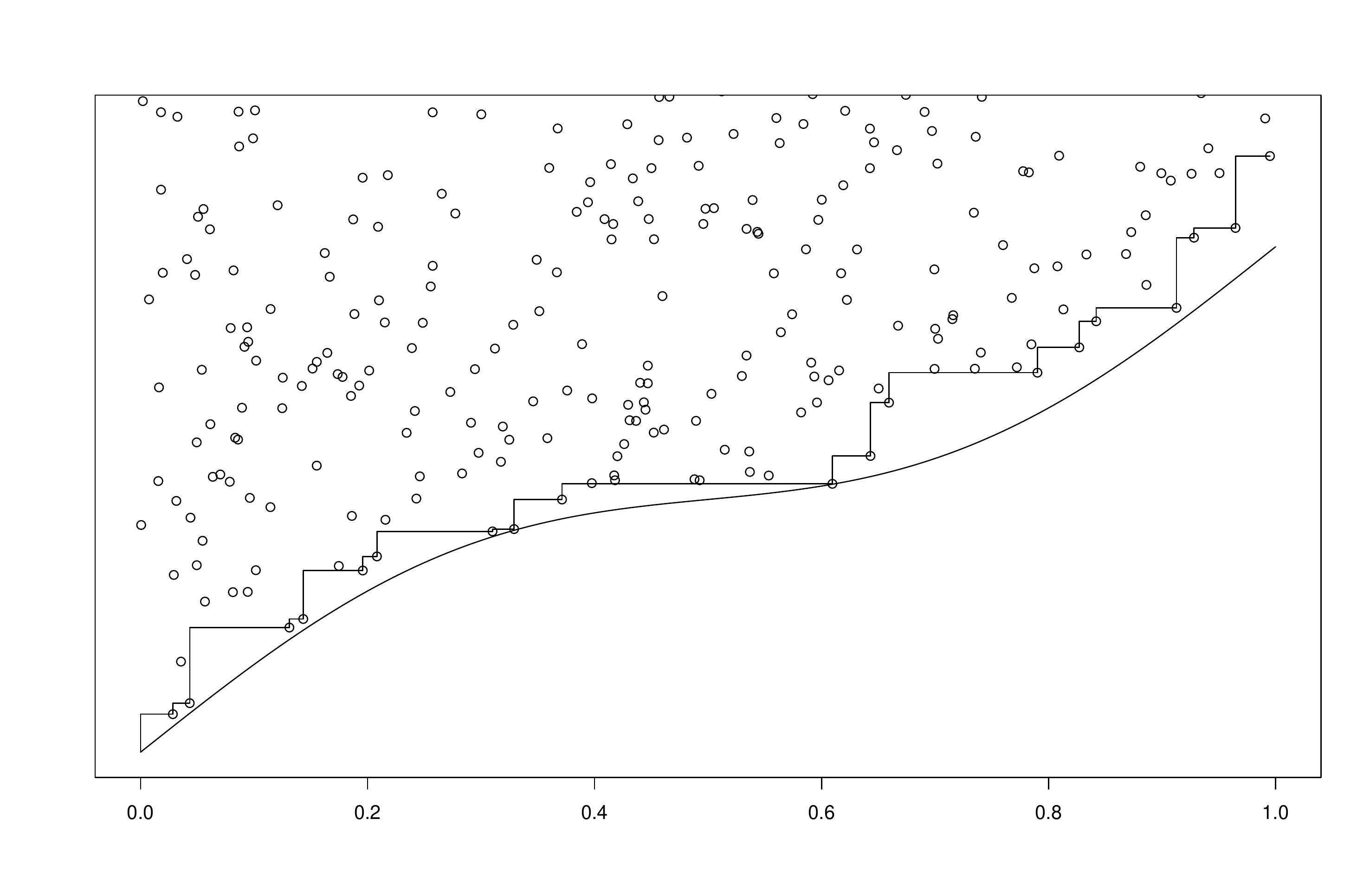}
\end{minipage}
\begin{minipage}[t]{0.49\textwidth}
\includegraphics[width=\textwidth]{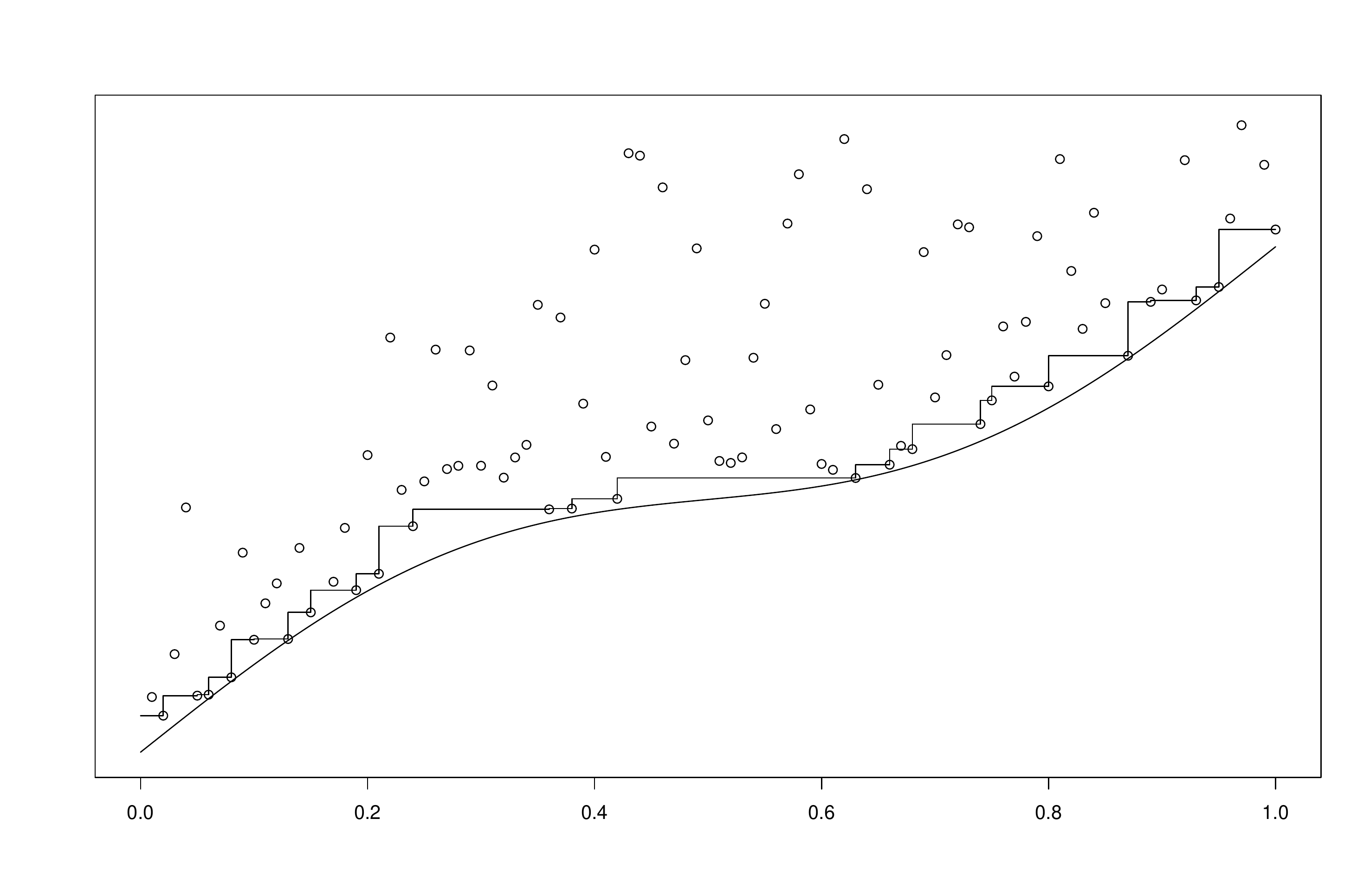}
\end{minipage}\end{center}\vspace{-10mm}
\caption{Construction of the estimator $\hat{g}^{Mon}$ in the PPP model ($n=100$, left) and in the regression-type model ($n=100$, $\eps_i\sim\Exp(1)$, right).}\label{fig-est-mon}
\end{figure}

\begin{proposition}
The nonparametric MLE $(\hat g^{Mon}(x),x\in[0,1))$ is a sufficient and complete statistics for ${\cal M}$.
\end{proposition}

\begin{proof}
Sufficiency follows again from the likelihood representation
\[{\cal L}(g)=\exp\Big(n+\sum_{j\ge 1} (-Y_j)_+\Big)\exp\Big(n\int_0^1g(x)\,dx\Big){\bf 1}\Big(g\le\hat g^{Mon}\Big),\]
using $g\in L^1$ because of $g(x)\in[g(0),g(1)]$ by monotonicity,
and the factorisation criterion.
For completeness we equip $\cal M$ with the ball $\sigma$-algebra for the uniform norm, cf. Examples 1.7.3, 1.7.4 in \cit{vanderVaart}, which is generated by the point evaluations $f\mapsto f(x)$, $x\in[0,1)$. In particular, this implies that $\hat g^{Mon}$ is measurable because its point evaluations are measurable. For fixed  $x_0\in[0,1)$, $y_0\in\R$ we have the bracket representation
\[ \{g\in{\cal M}\,|\, g(x_0)\ge y_0\}=\bigcup_{n\in\N} [y_0-n{\bf 1}_{[0,x_0)}(x),\infty).\]
Noting that maxima of monotone functions are again monotone, the brackets form an $\cap$-stable generator of the ball $\sigma$-algebra.
The proof now follows exactly that of Proposition \ref{PropSuff1}.
\end{proof}

In analogy with the $\cc^\beta(R)$-case we build the estimator
\[ \hte^{Mon}:=\int_0^1 \hat g^{Mon}(x)w(x)\,dx-\frac1n\sum_{j\ge 1}{\bf 1}\big(\hat g^{Mon}(X_j)=Y_j\big)w(X_j)\]
that will enjoy similar nice properties. We have to consider, however, weight functions $w$ whose support stays away from $x=1$ in order to avoid problems arising from $\hat g^{Mon}(x)\uparrow\infty$ as $x\uparrow 1$.

\begin{theorem}
Assume $\supp(w)\subset [0,1)$.
Then the estimator $\hte^{Mon}$ is for each finite sample size $n$ UMVU over the class ${\cal M}$ with
\[ \Var(\hte^{Mon})=\frac 1n\int_0^1\E[\hat g^{Mon}(x)-g(x)]w(x)^2dx.\]
For bounded $w$ it satisfies
\[ \Var(\hte^{Mon})\le \Big(\tfrac{3\pi(g(1)-g(0))}{2}\Big)^{1/2}\norm{w}_{\infty}^2 n^{-3/2}+O(n^{-2}).\]
\end{theorem}

\begin{remark}\mbox{}
\begin{enumerate}
\item  For $w\in L^p$ with $p>4$  the proof below still yields the rate $n^{-1/2}$ using H\"older's inequality instead of a supremum norm bound. For monotonous $g$ with bounded weak derivative $g'$, i.e. Lipschitz-continuous $g$, the asymptotic constant turns out to be exactly $\sqrt{\pi/2}\int_0^1w(x)^2\sqrt{g'(x)}\,dx$ by a dominated convergence argument.

\item Concerning the support of $w$, the proof shows that the remainder $O(n^{-2})$ is in fact  $2n^{-2}(1-\sup\{x:w(x)\not=0\})^{-1})$ and for varying weight functions $w_n$ we may allow a shrinking distance $\eps_n$ of $\supp(w_n)$ from $1$ such that $\eps_n n^{1/2}\to\infty$, implying a negligible order compared to $n^{-3/2}$.

\item The rate $n^{-3/2}$ is minimax optimal for the mean squared error over the class ${\cal M}$. This follows by adapting the proof of Theorem \ref{ThmLower} for the Lipschitz case $\beta=1$. We may just add to $g(x)$ a linear slope $Ax$ with $A>2cR$ such that $g\in {\cal M}$ holds for any realisation of $(\eps_k)$.
\end{enumerate}
\end{remark}

\begin{proof}
The proof follows along the lines of the proof for Theorem \ref{TheoMLE1}. Here the weighted counting process is
\[\bar N(t)=\sum_{j\ge 1}{\bf 1}\Big(Y_j\le t\wedge \min_{i:X_i\ge X_j}Y_i\Big)w(X_j), \quad t\in\R.\]
Its intensity is $\bar\lambda_t=n\int_0^1\int_{[g(x),t]} {\bf 1}(\min_{i:X_i\ge x} Y_i\ge s)dsdx$ and compensation yields the corresponding martingale $\bar M(t)$. The same limiting arguments, by restriction to the support of $w$, then yield again $\E[\hte^{Mon}]=\theta$ and
\[\Var(\hte^{Mon})=\frac 1n\int_0^1 \E[\hat g^{Mon}(x)-g(x)]w(x)^2dx.\]
It remains to estimate the last expectation. Suppose $w(x)=0$ for $x>1-\eps$ and some $\eps>0$ and let
\[g'_\infty(x):=\sup_{0<h\le 1-x}\frac{g(x+h)-g(x+)}{h}\in[0,\infty],\quad g(x+)=\lim_{y\downarrow x}g(y),\,x\in[0,1-\eps],\]
be the maximal function for the measure-valued derivative of $g$. Then for $x\in[0,1-\eps]$ with $g_\infty'(x)\in(0,\infty)$
\begin{align*}
\sqrt{n}\E[\hat g^{Mon}(x)-g(x)]
&=\int_0^\infty P(\hat g^{Mon}(x)-g(x)\ge sn^{-1/2})\,ds\\
&=\int_0^\infty\exp\Big(-n\int_x^1(sn^{-1/2}+g(x)-g(\xi))_+d\xi\Big)\,ds\\
&=\int_0^\infty\exp\Big(-\int_0^{n^{1/2}(1-x)}\Big(s+n^{1/2}(g(x)-g(x+n^{-1/2}u))\Big)_+du\Big)\,ds\\
&\le\int_0^\infty\exp\Big(-\int_0^{n^{1/2}(1-x)\wedge s/g'_\infty(x)}(s-g'_\infty(x)u)\,du\Big)\,ds\\
&\le\int_0^\infty\exp\Big(-s(n^{1/2}(1-x)\wedge s/g'_\infty(x))/2 \Big)\,ds\\
&\le(\pi g'_\infty(x)/2)^{1/2}+2n^{-1/2}\eps^{-1}e^{-n\eps^2/(2g'_\infty(x))},
\end{align*}
which trivially continues to hold with obvious extension if $g'_\infty(x)\in\{0,\infty\}$. We shall now establish a weak-$L^1$-estimate for $g_\infty'$ by adapting and improving (in the constant) classical results  \cite[Thm. 7.4]{rudin}.  For $\zeta>0$ we  define
\[ B_\zeta=\{x\in[0,1-\eps]\,|\,g'_\infty(x)\ge \zeta\}\]
and we shall prove $\abs{B_{\zeta}}\le \zeta^{-1}\norm{g}_{BV}$ with $\abs{B_\zeta}$ denoting the Lebesgue measure of $B_\zeta$, $\norm{g}_{BV}=g(1)-g(0)$.

To this end we construct a family $(x_i,h_i)_{i\in J}$, $J$ some countable set,  in $B_\zeta\times\R^+$ with $g(x_i+h_i)-g(x_i)\ge\zeta h_i$ and $[x_i,x_i+h_i)\cap [x_j,x_j+h_j)=\varnothing$ for all $i\not=j$ such that $\bigcup_{i\in J}[x_i,x_i+h_i)\supset B_\zeta$ holds. We proceed by (transfinite) recursion: since $x\mapsto g(x+)$ is right-continuous, so is $g'_\infty$ and thus  $x_0:=\inf B_\zeta$ lies in $B_\zeta$ such that there is some $h_0>0$ with $g(x_0+h_0)-g(x_0)\ge\zeta h_0$. Then define $x_1:=\min (B_\zeta\setminus[x_0,x_0+h_0))$, which is again in $B_\zeta$ by right-continuity such that some $h_1>0$ exists with $g(x_1+h_1)-g(x_1)\ge\zeta h_1$ and so on. Having thus defined $(x_i,h_i)_{i\in I'}$ for all (possibly infinite) ordinal numbers $I'$ smaller than some given ordinal $I$ we add $x_I=\inf (B_\zeta\cap\bigcap_{i\in \cup I'}[x_i+h_i,1-\eps])\in B_\zeta$ with some corresponding $h_I$ until $B_\zeta$ is exhausted by $\bigcup_{i\in I}[x_i,x_i+h_i)$. Then we just estimate
\[ g(1)-g(0)\ge \sum_{i\in J}g(x_i+h_i)-g(x_i)\ge \zeta \sum_{i\in J}h_i\ge \zeta\abs{B_\zeta}.\]

Using $\E[\hat g^{Mon}(x)-g(x)]\le (\pi g_\infty'(x)/(2n))^{1/2}+2n^{-1}\eps^{-1}$  and  for $a>0$ the integral bound
\[\int_0^{1-\eps} g_\infty'(x)^{1/2}dx\le \int_0^{1/a} z^{-1/2}d\abs{B_{z^{-1}}}+a^{1/2}(1-\eps)\le \norm{g}_{BV}\int_0^{1/a}z^{-1/2}dz+a^{1/2},\]
derived from $\abs{B_{z^{-1}}}\le z\norm{g}_{BV}$, we obtain with $a=\norm{g}_{BV}$
\begin{align*}
 \int_0^{1-\eps}\E[\hat g^{Mon}(x)-g(x)]\,dx
 &\le  (3\pi\norm{g}_{BV})^{1/2}(2n)^{-1/2}+O(n^{-1})
 \end{align*}
for $g\in{\cal M}$. The assertion follows by pulling $\norm{w}_\infty$ out of the integral.
\end{proof}

In the regression-type model the nonparametric MLE over $\cal M$ is likewise $\hat g^{Mon-regr}(x)=\min_{i:x\le i/n}\y_i$. Then  by the same arguments
\[ \hat\theta_n^{Mon-regr}:=\frac1n\sum_{i=1}^{n}\Big( \hat g^{Mon-regr}(i/n)-\lambda^{-1}{\bf 1}\Big(\hat g^{Mon-regr}(i/n)=\y_i\Big)\Big)w(i/n)\]
is an unbiased estimator of $\frac1{n}\sum_{i=1}^{n}g(i/n)w(i/n)$ under $\Exp(\lambda)$-noise with
$\Var(\hat\theta_n^{Mon-regr})=\frac{1}{n^2\lambda}\sum_{i=1}^{n}\E[\hat g^{Mon-regr}(i/n)-g(i/n)]w(i/n)^2$. Note that at the right end-point
$\E[\hat g^{Mon-regr}(1)-g(1)]=\lambda^{-1}$ holds, but that summand  only contributes $(n\lambda)^{-2}w(1)^2$ to the total variance which is usually negligible.

\section{Discussion}

An important application for the estimation of functionals are orthogonal series estimators, also called projection estimators. Let $(\phi_m)_{m\ge 1}$ be an orthonormal basis of $L^2([0,1])$. Then we can form the estimator $\hat g_M=\sum_{m=1}^M \hat\theta_m\phi_m$ of $g$ where $\hat\theta_m$ estimates the coefficient $\scapro{g}{\phi_m}_{L^2}$, i.e. $w=\phi_m$ in our notation. Using our estimators for $g\in\cc^\beta(R)$ we thus obtain as stochastic error in the $L^2$-risk:
\[ \E\Big[\norm{\hat g_M-\E[\hat g_M]}_{L^2}^2\Big]=\sum_{m=1}^M \Var(\hat\theta_m)\lesssim Mn^{-(2\beta+1)/(\beta+1)}.\]
For $L^2$-Sobolev spaces $H^s$ of regularity $s$ and standard bases like (trigonometric) polynomials, splines or wavelets we have the bias bound
$\sum_{m>M}\scapro{g}{\phi_m}^2\lesssim M^{-2s}$.  We always have $g\in \cc^\beta(R)\Rightarrow g\in H^\beta$ such that
\[  \E[\norm{\hat g_M-g}_{L^2}^2]\lesssim M^{-2\beta}+Mn^{-(2\beta+1)/(\beta+1)}\thicksim n^{-2\beta/(\beta+1)}\text{ for } M\thicksim n^{1/(\beta+1)}
\]
follows. This seems to be the first rate-optimal estimation result for series estimators in one-sided regression, cf. \cit{girardjacob}, \cit{jirakmeisterreiss}  for (optimal) rates and other approaches in the literature. We may, of course, also have $g\in H^s$ for some $s>\beta$, but then the derived rate is slower than the optimal $n^{-2s/(2s+1)}$. The unbiased estimation method essentially relies on a uniform control of the variation of $g$ and we do not know whether similar results can be obtained for Sobolev (or Besov)  instead of H\"older balls.

Concerning the function class ${\cal G}$ over which the nonparametric MLE is feasible and for which the derived estimator of $\theta$ exhibits  nice non-asymptotic properties, it was only essential for the stopping arguments as well as the completeness property that constants lie in $\cal G$ and that for $g_1,g_2\in{\cal G}$ also $g_1\wedge g_2$ and $g_1\vee g_2$ are in ${\cal G}$. Thus also ${\cal G}={\cal M}\cap \cc^\beta(R)$ or extensions to the multivariate case ${\cal G}=\cc_d^\beta(R)=\{g:[0,1]^d\to\R\,|\, \abs{g(x)-g(y)}\le R\abs{x-y}^\beta\}$ are possible. For smoothness degrees $\beta>1$ or other shape constraints like convexity our method does not transfer directly, but may possibly be adapted, see e.g. \cit{baldinreiss}, where also the intensity $\lambda$ is estimated.


\begin{figure}[t]\begin{center}
\begin{minipage}[t]{0.49\textwidth}
\includegraphics[width=\textwidth]{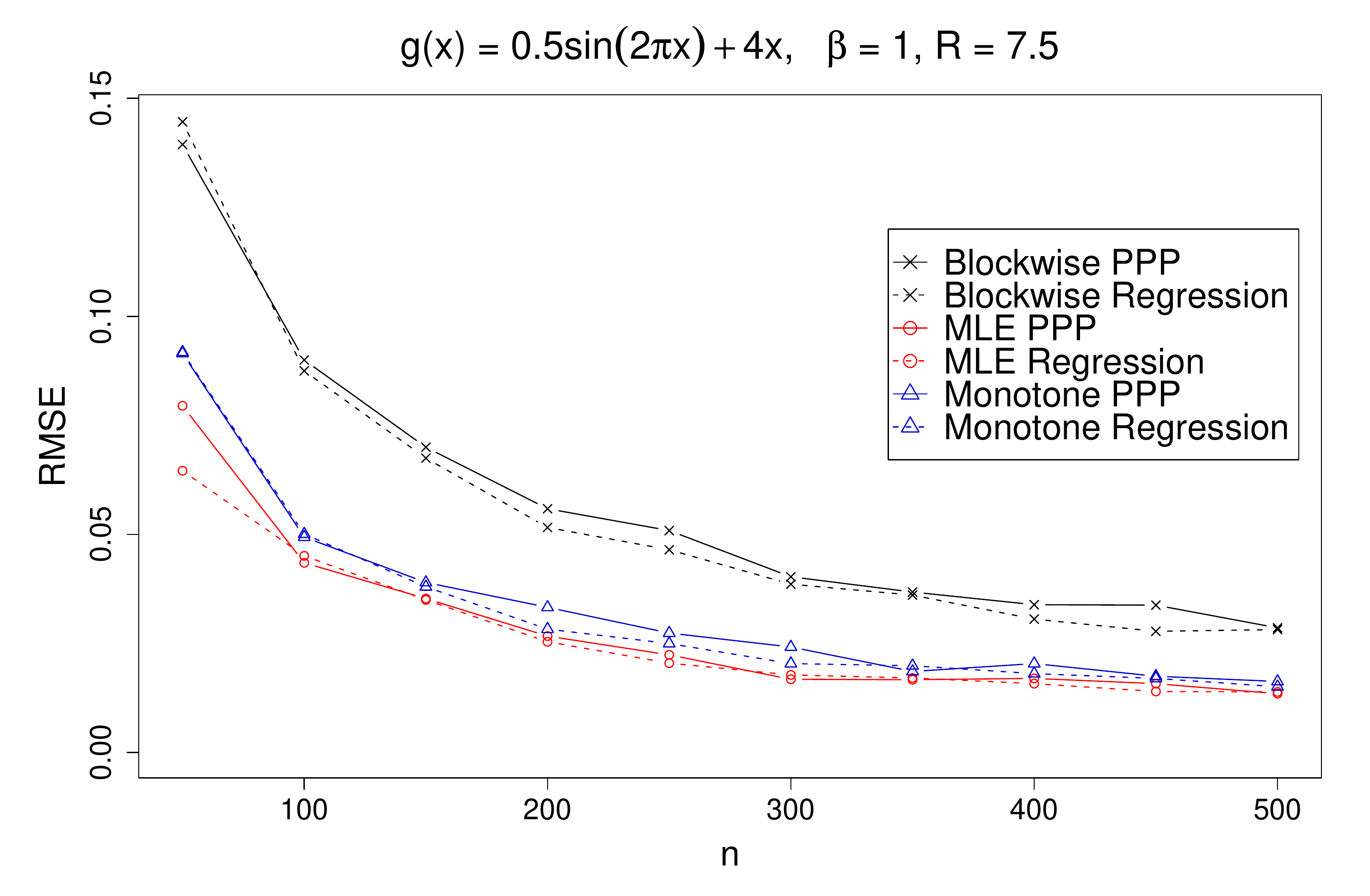}
\end{minipage}
\begin{minipage}[t]{0.49\textwidth}
\includegraphics[width=\textwidth]{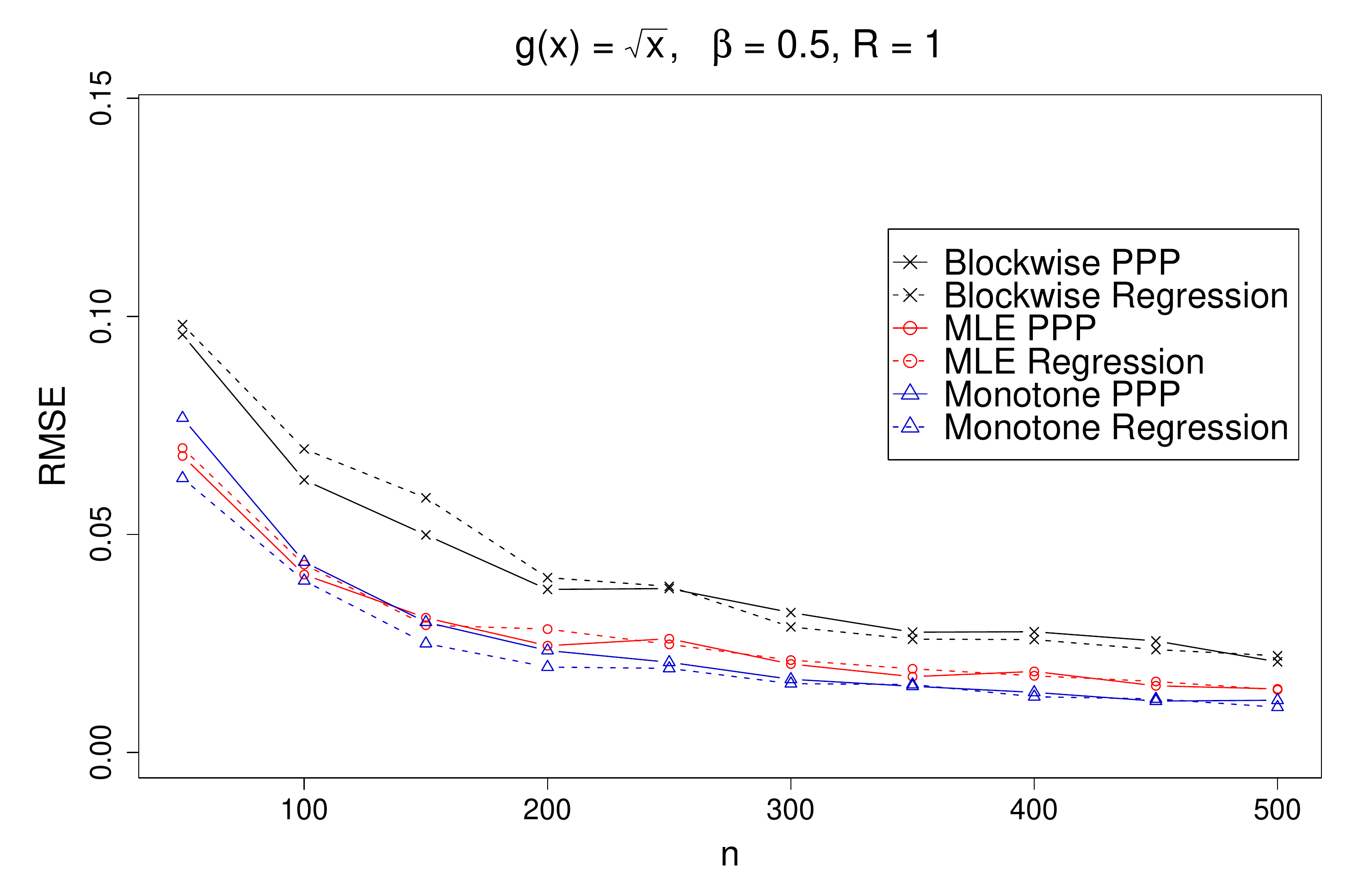}
\end{minipage}\end{center}\vspace{-10mm}
\caption{Monte-Carlo errors for the different estimators and two functions $g$}\label{fig-sim}\end{figure}

Finally, in a small simulation example we investigate the behaviour of the blockwise estimator, the MLE and the monotone MLE for $\theta=\int_0^1g(x)dx$ on finite samples. We simulate the PPP model as well as the regression-type model with $\Exp(1)$-distributed noise. For two different monotone regression functions $g$ the RMSE (root mean squared error) is  estimated in $M=200$ Monte Carlo repetitions.
On the left-hand side of Figure \ref{fig-sim} the RMSE results for $g(x)=0.5\sin(2\pi x)+4x$ are shown and on the right-hand side those for $g(x)=\sqrt x$. It can be seen that all three estimators work well even for the small sample size  $n=50$ and that their performances in the PPP and the regression model are comparable.

The blockwise estimator does not perform so much worse than the ML estimators. From our theoretical results this is to be expected: the ratio of the upper bounds for the nonasymptotic variance of $\hat\theta^{MLE}$ and of $\hat\theta^{block}$ is given by
\[ \frac{\Gamma(\beta/(\beta+1))\beta(\beta+1)^{-1/(\beta+1)}}{\beta^{-\beta/(\beta+1)}(\beta+1)}=\Gamma(\beta/(\beta+1))\beta^{(2\beta+1)/(\beta+1)}(\beta+1)^{-(\beta+2)/(\beta+1)},\]
which approaches one for $\beta\downarrow 0$, has a minimum $0.54$ at $\beta\approx 0.47$ and then increases to about $0.63$ for $\beta\uparrow 1$. Note, however, that both upper bounds are not tight. Since the simple blockwise approach is faster to compute, which is particularly relevant for any adaptive estimator, and is theoretically easier to analyse than the MLE (especially for the CLT, but also for an adaptive procedure), we conclude that both approaches are attractive and should be considered in their own right.


\section{Appendix}
\subsection{Technical results}

We formulate a  stopping theorem for continuous-time martingales, which does not seem readily available in the literature.

\begin{lemma}\label{LemStop}
Let $(M(t),t\ge t_0)$ be a c\`adl\`ag martingale with $M(t_0)=0$ and let $\tau$ be a stopping time with values in $[t_0,\infty]$, both on some filtered probability space. If $\E[\langle M\rangle_\tau]$ is finite, then $\E[M(\tau)]=0$ and $\E[M(\tau)^2]=\E[\langle M\rangle_\tau]$ hold.
\end{lemma}


\begin{proof}
From the Burkholder-Davis-Gundy inequality (Thm. 26.12 in \cit{kallenberg}) and the identity $\E[[M]_\tau]=\E[\langle M\rangle_\tau]$ (e.g. by Prop. 4.50(c) in \cit{jacodshiryaev} for $M^\tau$), we conclude $\E[\sup_{t\ge t_0}M_{t\wedge\tau}^2]\lesssim \E[\langle M\rangle_\tau]$. Hence, $(\abs{M_{t\wedge\tau}}^p)_{t\ge t_0}$, $p\in\{1,2\}$, is uniformly integrable and by optional stopping $\E[M_\tau]=\lim_{t\to\infty}\E[M_{\tau\wedge t}]=0$ follows as well as
$\E[M_\tau^2]=\E[[M]_\tau]=\E[\langle M\rangle_\tau]$.
\end{proof}

A moment bound for the stopping time in the proof of Theorem \ref{ThmRegr} is provided.

\begin{lemma}\label{LemTauMom}
Under the assumptions of Theorem \ref{ThmRegr} we have for $\tau=\mY+Rh^\beta$
\[\E[(\tau-g(i/n))^p]^{1/p}\lesssim Rh^\beta+(n\lambda h)^{-1}\]
as $nh\to\infty$ for any $p>0$.
\end{lemma}

\begin{proof}
The property $\mY\le \max_{i\in\tilde I_k}g(i/n)+\min_{i\in\tilde I_k}\eps_i$ implies for $nh\to\infty$
\[ P\Big(n\lambda h(\mY-\max_{i\in\tilde I_k}g(i/n))\ge z\Big)\le \bar F_\eps(z/n\lambda h)^{nh}=e^{nh\log\bar F_\eps(z/n\lambda h)}\to e^{-z}.
\]
Using $\bar F_\eps(z/nh)^{nh}\lesssim (1+z/nh)^{-nh\rho}$, we establish
\[\lim_{R\to\infty}\sup_{n,h}\int_R^\infty z^{p-1}P\Big(n\lambda h\Big(\mY-\max_{i\in\tilde I_k}g(i/n)\Big)\ge z\Big)\,dz= 0\]
 for any $p\ge 1$ such that by uniform integrability
\[\limsup_{nh\to\infty}\E\Big[\Big(n\lambda h\abs{\mY-\max_{i\in\tilde I_k}g(i/n)}\Big)^p\Big]\le \int_0^\infty z^pe^{-z}dz<\infty\]
follows.
By the H\"older condition $g$ varies at most by $Rh^\beta$ on each block and thus
$\E[(\tau-\min_{i\in\tilde I_k}g(i/n))^p]\lesssim (Rh^\beta+(n\lambda h)^{-1})^p$ holds.
\end{proof}

We need the following interesting self-normalising property. The constant is certainly not optimal.

\begin{lemma} \label{LemSelfnorm}
Suppose that a non-negative random variable $X$ satisfies $P(X\ge x)=e^{-a(x)}$, $x\ge 0$, with a strictly increasing convex function $a$. Then $\E[X^2]\le 6(e+1)\E[X]^2$ holds.
\end{lemma}

\begin{proof}
The property $P(a(X)\ge a(x))=e^{-a(x)}$ shows that $Y:=a(X)$ is $\Exp(1)$-distributed. Since the inverse $a^{-1}$ of $a$ exists, we may use $a^{-1}(0)=0$ and the concavity of $a^{-1}$ to calculate
\begin{align*}
\E[X]^2 &= \E[a^{-1}(Y)^2]=\int_0^\infty\int_0^\infty a^{-1}(x)a^{-1}(y)e^{-x-y}dydx\\
&= \int_0^\infty\int_0^z a^{-1}(x)a^{-1}(z-x)\,dx\,e^{-z}dz\\
&\ge  \int_0^\infty\int_0^z \frac{x}{z}a^{-1}(z)\frac{z-x}{z}a^{-1}(z)\,dx\,e^{-z}dz= \int_0^\infty \frac{z}{6}a^{-1}(z)^2e^{-z}dz.
\end{align*}
By monotonicity, we have $\int_1^2za^{-1}(z)^2e^{-z}dz\ge e^{-1}\int_0^1a^{-1}(z)^2e^{-z}dz$. This shows
\[ 6\E[X]^2\ge \frac{1}{e+1}\int_0^\infty a^{-1}(z)^2e^{-z}dz=\frac{1}{e+1}\E[X^2].\]
\end{proof}

\subsection{Proof of Theorem \ref{ThmLepski}}

For $h_m< h^\ast$ we infer from the deviation bound in Proposition \ref{Propcritval} below
\begin{align*}
&P(\hat h=h_m)\le \sum_{m'=1}^{m-1}\Big(P(\abs{\tilde\theta_{n,h_{m'}}^{block}-\theta^{(n)}}>\kappa_{m'})+P(\abs{\tilde\theta_{n,h_{m+1}}^{block}-\theta^{(n)}}>\kappa_{m+1})\Big)\\
&\le \sum_{m'=1}^{m-1}\Big(4n^{-2c}+n\bar F_\eps((c\log n-2)/(nh_{m'}))^{nh_{m'}}+n\bar F_\eps((c\log n-2)/(nh_{m+1}))^{nh_{m+1}}\Big)\\
&\lesssim M(n^{-2c}+n^{1-\underline\lambda c})
\end{align*}
for $\underline\lambda\in(0,\lambda)$ and $n$ sufficiently large. By Cauchy-Schwarz inequality we thus infer
\[ \E[(\tilde\theta_{n}^{block}-\theta^{(n)})^2{\bf 1}(\hat h<h^\ast)]\lesssim \E[(\tilde\theta_{n}^{block}-\theta^{(n)})^4]^{1/2} M(n^{-2c}+n^{1-\underline\lambda c})^{1/2}.
\]
From the exponential moment bound \eqref{EqExpMoment} and Lemma \ref{LemTauMom} we obtain that the fourth moment of the error remains bounded (even tends to zero) such that the first inequality follows.

By construction, we have for $h_{\hat m}:=\hat h>h^\ast=:h_{m^\ast}$ that $\abs{\tilde\theta_{n}^{block}-\tilde\theta_{n,h^\ast}^{block}}\le \kappa_{\hat m}+\kappa_{m^\ast}$ holds. Using $H_x(y)\approx xy^2$ and $h_m\gtrsim (\log n)^2n^{-1}$, we obtain
\[ \kappa_{\hat m}+\kappa_{m^\ast}\lesssim \frac{(\log n)^2}{(nh^\ast)^2}+\frac{\sqrt{\log n}}{n\sqrt{h^\ast}}\Big(1+\max_{h_m\ge h^\ast}h_m\sum_{k=0}^{h_m^{-1}-1}\#\{i\in\tilde I_{k,h_m}:\y_i\le \y_{k,h}^\ast+(nh_m)^{-1}\}\Big).
\]
For each fixed $h_m$ we have by  compensation of the block-wise counting process
\[ \E\Big[\Big(\sum_{k=0}^{h_m^{-1}-1}\#\{i\in\tilde I_{k,h_m}:\y_{k,h_m}^\ast<\y_i\le \y_{k,h_m}^\ast+(nh_m)^{-1}\}\Big)^2\Big]\le h_m^{-1}\sum_{k=0}^{h_m^{-1}-1}\E[A_k^2+A_k]\]
with
\[A_k=\sum_{i\in\tilde I_{k,h_m}}\int {\bf 1} \Big(\y_{k,h_m}^\ast<s+g(i/n)\le \y_{k,h_m}^\ast+(nh_m)^{-1}\Big) \frac{f_\eps(s)}{\bar F_\eps(s)}\,ds\le \norm{f_\eps/\bar F_\eps}_\infty\thicksim 1.\]
Since by definition $\#\{i\in\tilde I_{k,h_m}:\y_i\le \y_{k,h_m}^\ast\}=1$ a.s., we have
\[ \E\Big[\Big(h_m\sum_{k=0}^{h_m^{-1}-1}\#\{i\in\tilde I_{k,h_m}:\y_i\le \y_{k,h_m}^\ast+(nh_m)^{-1}\}\Big)^2]\lesssim 1.
\]
A (crude) bound for the maximum via the sum thus yields the second inequality:
\[ \E[(\tilde\theta_{n}^{block}-\tilde\theta_{n,h^\ast}^{block})^2{\bf 1}(\hat h\ge h^\ast)]\le \E[(\kappa_{\hat m}+\kappa_{m^\ast})^2{\bf 1}(\hat h>h^\ast)]\lesssim \frac{(\log n)^4}{(nh^\ast)^4}+\frac{M\log n}{n^2 h^\ast}.
\]
For the asymptotic rate just note that the geometric grid of bandwidths suffices to achieve $h^\ast\thicksim n^{-1/(\beta+1)}$ asymptotically such that inserting $\E[(\tilde\theta_{n,h^\ast}^{block}-\theta^{(n)})^2] \lesssim n^{-(2\beta+1)/(\beta+1)}+n^{-4\beta/(\beta+1)}$ from Theorem \ref{ThmRegr} and the triangle inequality yield the result, noting that the risk on $\{\hat h<h^\ast\}$ is negligible due to the choice of $c$.
It remains to prove the following deviation inequality.

\begin{proposition}\label{Propcritval}
For any $h,x,\kappa>0$ with $Rh^\beta\le (nh)^{-1}$, $2xh^{1/2}\norm{w}_\infty<1$ and $\kappa<(\delta -2Rh^{\beta})nh$ we have with probability at least $1-2e^{-2x^2}-h^{-1}\bar F_\eps(\kappa/(nh))^{nh}$ the bound
\begin{align*}
n\lambda h^{1/2}\abs{\tilde\theta_{n,h}^{block}-\theta^{(n)}}&\le \sum_{i=1}^n{\bf 1}(\y_i\le \tau^{(i)})H_x(h^{1/2}w(i/n))+C^2(\kappa+2)^2n^{-1}h^{-3/2} \norm{w}_{1}+x.
\end{align*}
\end{proposition}

\begin{proof}
We consider  the  martingale $M(t)$ in \eqref{EqMartRegr} and the associated stopping rule $\tau$. By the substitution rule \cite[Thm. 26.7]{kallenberg} we obtain for $\gamma>-1/\norm{w}_\infty$ the exponential (local) martingale
\begin{align*}
{\cal E}(t) &=\exp\Big(\sum_{i\in\tilde I_k}{\bf 1}(\y_i\le t)\log(1+\gamma w(i/n))+\log(\bar F_\eps(\y_i\wedge t-g(i/N)))\gamma w(i/n)\Big)\\
&= \exp\Big(\gamma M(t)-\sum_{i\in\tilde I_k}{\bf 1}(\y_i\le t)\Big(\gamma w(i/n)-\log(1+\gamma w(i/n))\Big) \Big).
\end{align*}
We infer from $\abs{\tilde I_k}<\infty$ and the fact that $\bar F_\eps(\eps_i)\sim U([0,1])$ has finite $p$-moments for all $p>-1$ via Lemma \ref{LemStop} the stopping result
\begin{equation}\label{EqExpMoment}
 \E\Big[\exp\Big(\gamma M(\tau)-\sum_{i\in\tilde I_k}{\bf 1}(\y_i\le \tau)\Big(\gamma w(i/n)-\log(1+\gamma w(i/n))\Big) \Big)\Big]=1.
\end{equation}
Using representation \eqref{EqErrorMart}, the independence among blocks  yields $\E[e^{\gamma Z_\gamma}]=1$ for
\begin{align*}
Z_\gamma := &n\lambda(\tilde\theta_{n,h}^{block}-\theta^{(n)})-\sum_{i=1}^nG_\eps(\y_i\wedge\tau^{(i)}-g(i/n)) w(i/n)\\
& -\sum_{i=1}^n{\bf 1}(\y_i\le \tau^{(i)})\Big(w(i/n)-\frac{\log(1+\gamma w(i/n))}{\gamma}\Big).
\end{align*}
We choose $\gamma=\pm 2x h^{1/2}$ and obtain
by Markov inequality  $P( h^{1/2} Z_{2x h^{1/2}}\ge x)\le e^{-2x^2}$, $P( h^{1/2} Z_{-2x h^{1/2}}\le- x)\le e^{-2x^2}$ such that with probability $1-2e^{-2x^2}$
\begin{align}
&\babs{n\lambda h^{1/2}(\tilde\theta_{n,h}^{block}-\theta^{(n)})-h^{1/2}\sum_{i=1}^nG_\eps(\y_i\wedge\tau^{(i)}-g(i/n)) w(i/n)}\nonumber\\
&\qquad \le \sum_{i=1}^n{\bf 1}(\y_i\le \tau^{(i)})H_x(h^{1/2}w(i/n))+x.\label{EqHP1}
\end{align}

From $\abs{G_\eps(z)}\le C^2 z^2$ for $z\in[0,\delta]$ and $Rh^\beta\le (nh)^{-1}$ we infer
\begin{align*}
&P\Big(\max_{i\in\tilde I_k}\abs{G_\eps(\y_i\wedge \tau-g(i/n))}>C^2(\kappa+2)^2/(nh)^2\Big)\\
&\le P\Big(\max_{i\in\tilde I_k}(\y_i\wedge \tau-g(i/n))>(\kappa+2)/(nh)\Big)\\
&\le P\Big(\min_{i\in\tilde I_k}\eps_i+2(nh)^{-1}>(\kappa+2)/(nh)\Big)= \bar F_\eps\Big(\kappa/(nh)\Big)^{nh}.
\end{align*}
We thus obtain with probability $1-\bar F_\eps(\kappa/(nh))^{nh}$ the bound
\[ \babs{\sum_{i\in\tilde I_k}G_\eps(\y_i\wedge\tau-g(i/n))w(i/n)}\le \frac{C^2(\kappa+2)^2}{(nh)^2} \sum_{i\in\tilde I_k}\abs{w(i/n)}.\]
Summing over the $h^{-1}$ blocks implies with probability $1-h^{-1}\bar F_\eps(\kappa/(nh))^{nh}$
\[ h^{1/2}\babs{\sum_{i=1}^n G_\eps(\y_i\wedge\tau^{(i)}-g(i/n))w(i/n)}\le C^2(\kappa+2)^2n^{-1}h^{-3/2}\norm{w}_{1}.\]
In view of \eqref{EqHP1} this yields the result.
\end{proof}

\subsection{Proof of Theorem \ref{ThmCLT1}}

Let $r_n\to 0$ such that $r_n^{3}n\to\infty$ and $r_n^{-1}\in\N$. On each block $J_l=[lr_n,(l+1)r_n)$, $l=0,\ldots,r_n^{-1}-1$, we can define the blockwise $\cc^1(R)$-MLE
\[ \hat g^{MLE}_l(x)=\min_{i:X_i\in J_l}(Y_i+R\abs{x-X_i}),\quad x\in J_l.\]
Note that by definition the blockwise MLE is at least as large as the global MLE, i.e. $\hat g^{MLE}_l\ge\hat g^{MLE}$.
By construction, $(\hat g^{MLE}_l)_l$ are independent and each
\[ \hat\theta_l^{MLE}:=\int_{J_l} \hat g_l^{MLE}(x)w(x)\,dx-\frac1n\sum_{j\ge 1}{\bf 1}\big(X_j\in J_l,\,\hat g_l^{MLE}(X_j)=Y_j\big)w(X_j)\]
enjoys the non-asymptotic properties of Theorem \ref{TheoMLE1} on $J_l$, in particular $\E[\hat\theta_l^{MLE}]=\int_{J_l}g(x)w(x)dx$ and $\Var(\hat\theta_l^{MLE})=\frac1n\int_{J_l}\E[\hat g_l^{MLE}(x)-g(x)]w(x)^2\,dx$. Let us therefore first establish for the blockwise MLE $\tilde\theta_n:=\sum_{l=0}^{r_n^{-1}-1} \hat\theta_l^{MLE}$ that
\begin{equation}\label{EqCLTtilde}
\Big(\frac1n\sum_{l=0}^{r_n-1}\int_{J_l}\E[\hat g_l^{MLE}(x)-g(x)]w(x)^2\,dx\Big)^{-1/2}(\tilde\theta_n-\theta)\Rightarrow N(0,1).
\end{equation}
By independence of $(\hat\theta_l^{MLE})$, for the CLT to hold it suffices to check the 4th moment Lyapunov condition
\[ \frac{\sum_{l=0}^{r_n^{-1}-1}\E[(\hat\theta_l-\theta_l)^4]}{\Var(\tilde\theta_n)^2}\to 0.\]
For each $l=0,\ldots,r_n-1$ let $(\bar M_{l,t})_t$ be the compensated weighted counting process from the proof of Theorem \ref{TheoMLE1}, restricted to $J_l$. The (non-predictable) quadratic variation  of  $(\bar M_{l,t})$ is given by the sum of squared jumps:
\[ [\bar M_l]_t=\sum_{s\le t} (\Delta\bar M_{l,s})^2=\sum_{j\ge 1}{\bf 1}\Big(X_j\in J_l, Y_j\le t\wedge \min_{i: X_i\in J_l}(Y_i+R\abs{X_j-X_i})\Big)w(X_j)^2.\]
The Burkholder-Davis-Gundy inequality (e.g. Thm 26.12 in \cit{kallenberg}) then yields by similar arguments as for $(\bar M_t)$ above
\begin{align*}
\E[\bar M_{l,\infty}^4]&\lesssim \E[[\bar M_l]_\infty^2]
= \E\Big[\Big(n\int_{J_l}(\hat g_l^{MLE}-g)w^2\Big)^2+n\int_{J_l}(\hat g_l^{MLE}-g)w^4\Big].
\end{align*}
Using Jensen's inequality, we find 
\begin{align*}
&\sum_{l=0}^{r_n^{-1}-1}\E[(\hat\theta_l^{MLE}-\theta_l)^4]
=\frac{1}{n^4}\sum_{l=0}^{r_n^{-1}-1}\E[\bar M_{l,\infty}^4]\\
&\quad\lesssim \sum_{l=0}^{r_n^{-1}-1}\E\Big[n^{-2}\Big(\int_{J_l}(\hat g_l^{MLE}-g)w^2\Big)^2+n^{-3}\int_{J_l}(\hat g_l^{MLE}-g)w^4\Big]\\
&\quad\le \sum_{l=0}^{r_n^{-1}-1}\Big(n^{-2}r_n\int_{J_l}\E[(\hat g_l^{MLE}-g)^2]w^4+n^{-3}\int_{J_l}\E[\hat g_l^{MLE}-g]w^4\Big).
\end{align*}
As in \eqref{EqgMLE} we can bound
\begin{equation}\label{EqgMLEr}
P(\hat g^{MLE}_l(x)-g(x)\ge s)\le \begin{cases}\exp(-nR(s/2R)^2),& s\in[0,2Rr_n],\\ \exp(-n(sr_n-Rr_n^{2})),& s>2Rr_n.\end{cases}
\end{equation}
Noting $r_n n^{1/2}\to\infty$ and $\norm{w}_\infty<\infty$, we  apply the  moment bound of Lemma \ref{LemSelfnorm} to $\hat g^{MLE}_l(x)-g(x)$ with $a(s)=n\int_{J_l}(s-R\abs{\xi-x}+g(x)-g(\xi))_+dx$ and integrate over $s$ to obtain
\[ \sum_{l=0}^{r_n^{-1}-1}\E[(\hat\theta_l^{MLE}-\theta_l)^4]
\lesssim (r_n+n^{-1/2})\big(n^{-3/2}\big)^2.
\]
Hence, in view of $\Var(\tilde\theta_n)\ge\Var(\hat\theta_n^{MLE})\thicksim n^{-3/2}$ the Lyapunov condition is satisfied and the CLT \eqref{EqCLTtilde} follows.

In the second step we show that the difference between $\tilde\theta_n$ and $\hat\theta_n^{MLE}$ is of small stochastic order $o_P(n^{-3/4})$. First, we note that the above martingale arguments yield
\[\E[(\tilde\theta_n-\hat\theta_n^{MLE})^2]=\Var(\tilde\theta_n-\hat\theta_n^{MLE})=n^{-1}\sum_{l=0}^{r_n^{-1}-1}\int_{J_l}\E[\hat g_l^{MLE}(x)-\hat g^{MLE}(x)]w(x)^2dx.\]
Introduce the notation $\hat g^{MLE}_{-l}(x)=\min_{i:X_i\notin J_l}(Y_i+R\abs{x-X_i}^\beta)$ and consider the event
\[\Omega_n=\Big\{\forall l=0,\ldots,r_n^{-1}-1\; \exists x\in J_l:\hat g_l^{MLE}(x)=\hat g^{MLE}(x)\Big\}
\]
whose complement is given by $\Omega_n^\complement=\bigcup_l \{\min_{x\in J_l}(\hat g_l^{MLE}-\hat g_{-l}^{MLE})(x)>0\}$. By independence of $\hat g_l^{MLE}$ and $\hat g_{-l}^{MLE}$ and conditioning on the latter we obtain
\begin{align*}
&P\Big(\min_{x\in J_l}(\hat g_l^{MLE}-\hat g_{-l}^{MLE})(x)>0\Big) = \E\Big[\exp\Big(-n\int_{J_l}(\hat g_{-l}^{MLE}-g)(x)\,dx\Big)\Big]\\
&\quad \le  \E\Big[\exp\Big(-nr_n\min\Big((\hat g_{-l}^{MLE}-g)(lr_n),(\hat g_{-l}^{MLE}-g)((l+1)r_n)\Big)\Big)\Big].
\end{align*}
Using $\hat g^{MLE}_{l'}\ge \hat g_{-l}^{MLE}$ for $l'\not=l$, the bound \eqref{EqgMLEr} yields
\begin{align*}
P\Big(\min_{x\in J_l}(\hat g_l^{MLE}-\hat g_{-l}^{MLE})(x)>0\Big) &\le 2\max_{l',x}\E\Big[\exp(-nr_n(\hat g^{MLE}_{l'}(x)-g(x)))\Big]\\
&\lesssim n^{-1}r_n^{-2}.
\end{align*}
We conclude $P(\Omega_n^\complement)=O(n^{-1} r_n^{-3})\to 0$  by a union bound and the choice of $r_n$.

On the event $\Omega_n$ the left-most point $L_l$ in $J_l$ where $\hat g_l^{MLE}$ and $\hat g^{MLE}$ coincide is well defined and satisfies for $l\ge 1$
\begin{align*}
L_l&:=\inf\{x\in J_l\,|\,\hat g_l^{MLE}(x)=\hat g^{MLE}(x)\}\\
&=\inf\{x\in J_l\,|\,\hat g_l^{MLE}(x)\le\hat g_{l-1}^{MLE}(lr_n)+R(x-lr_n)\}.
\end{align*}
Now $L_l=lr_n$ holds  on $\Omega_n$ if the corresponding right-most point $R_{l-1}:=\sup\{x\in J_{l-1}\,|\,\hat g_{l-1}^{MLE}(x)=\hat g^{MLE}(x)\}$ on $J_{l-1}$ satisfies $R_{l-1}<lr_n$
and vice versa $L_l>lr_n \Rightarrow R_{l-1}=lr_n$. Due to this symmetry we  only consider the case $L_l>lr_n$. For $z\in(0,r_n]$ and $l=1,\ldots,r_n^{-1}-1$ a rough bound yields:
\begin{align*}
P\Big(L_l\ge lr_n+z\,\Big|\,((X_i,Y_i){\bf 1}(X_i<lr_n))_{i\ge 1}\Big) &\le \exp\Big(-n\int_{lr_n}^{lr_n+z} (\hat g_{l-1}^{MLE}(x)-g(x))\,dx\Big)\\
\le \exp\Big(-nz (\hat g_{l-1}^{MLE}(lr_n)-g(lr_n))\Big).
\end{align*}
Since $\hat g_l^{MLE}(L_l)=\hat g_{l-1}^{MLE}(L_l)$ holds and both functions are in $\cc^1(R)$, we obtain the bound
\[
 \int_{lr_n}^{L_l}(\hat g_l^{MLE}-\hat g_{l-1}^{MLE})(x)dx\le \int_{lr_n}^{L_l} 2R(L_l-x) dx\lesssim (L_l-lr_n)^2.
\]
By the identity $\E[Z^2]=\int_0^\infty 2z P(Z\ge z)dz$ for non-negative random variables $Z$ and by the above probability bound for $L_l$  we  obtain further
\begin{align*}
\E\Big[\int_{lr_n}^{L_l}(\hat g_l^{MLE}-\hat g^{MLE}){\bf 1}_{\Omega_n}\Big] & \lesssim \E\Big[\int_0^{r_n} z e^{-nz(\hat g_{l-1}^{MLE}(lr_n)-g(lr_n))}dz\Big]\\
&\lesssim \E\Big[\min\Big(\frac{1}{n(\hat g_{l-1}^{MLE}(lr_n)-g(lr_n))},r_n\Big)^{2}\Big].
\end{align*}
Using \eqref{EqgMLEr} on $J_{l-1}$, we arrive, after suitable substitution inside the integral, at
\begin{align*} \E\Big[\int_{lr_n}^{L_l}(\hat g_l^{MLE}-\hat g^{MLE}){\bf 1}_{\Omega_n}\Big] &\lesssim \int_0^\infty \min(u^{-1}n^{-1},r_n^{2})e^{-u}du\\
&\lesssim n^{-1}\log(nr_n^2).
\end{align*}

Summing over $l$, bounding the alternative case $R_{l-1}<lr_n$ by the same estimate and using $\norm{w}_\infty<\infty$, we arrive at
\begin{equation}\label{EqIntghat}
  \E\Big[\sum_{l=0}^{r_n^{-1}-1}\int_{J_l}(\hat g_l^{MLE}-\hat g^{MLE})(x)w(x)^2dx\,{\bf 1}_{\Omega_n}\Big]\lesssim (nr_n)^{-1}\log(nr_n^2)=o(n^{-1/2}).
\end{equation}
This gives the desired result $\E[(\tilde\theta_n-\hat\theta_n^{MLE})^2{\bf 1}_{\Omega_n}]=o(n^{-3/2})$ with $P(\Omega_n)\to 1$.

 Furthermore, from \eqref{EqIntghat} we derive also that
\begin{align*}
&\sum_{l=0}^{r_n-1}\int_{J_l}\E[\hat g_l^{MLE}(x)-g(x)]w(x)^2\,dx-\int_0^1\E[\hat g^{MLE}(x)-g(x)]w(x)^2\,dx\\
&=o(n^{-1/2})+O\Big(\sup_{l,x\in J_l}\E[(\hat g_l^{MLE}(x)-\hat g^{MLE}(x)){\bf 1}_{\Omega_n^\complement}]\Big).
\end{align*}
By the Cauchy-Schwarz inequality, the last term is at most of order $O(n^{-1/2}P(\Omega_n^\complement)^{1/2})=o(n^{-1/2})$. Hence, applying Slutsky's Lemma twice to the CLT \eqref{EqCLTtilde}, we  arrive at
\[ \Big(\frac1n\int_{0}^1\E[\hat g^{MLE}(x)-g(x)]w(x)^2\,dx\Big)^{-1/2}(\hat\theta_n^{MLE}-\theta)\Rightarrow N(0,1).\]

By Jensen's inequality, $\Var(\int_{J_l} (\hat g_l^{MLE}-g)w^2)\le r_n\int_{J_l}\E[ (\hat g_l^{MLE}-g)^2]w^4\lesssim r_n^2n^{-2/3}$,
which implies
\[ \Var\Big(\sum_{l=0}^{r_n-1}\int_{J_l}(\hat g_l^{MLE}(x)-g(x))w(x)^2\,dx\Big)\lesssim r_nn^{-2/3}=o(n^{-2/3}).
\]
We infer $\sum_{l=0}^{r_n-1}\int_{J_l}(\hat g_l^{MLE}-g)w^2\xrightarrow{P} \sum_{l=0}^{r_n-1}\int_{J_l}\E[\hat g_l^{MLE}-g]w^2$.
Together with \eqref{EqIntghat}, this yields the  CLT
\[ \Big(\int_0^1(\hat g^{MLE}(x)-g(x))w(x)^2\,dx\Big)^{-1/2}n^{-1/2}(\hat\theta_n^{MLE}-\theta)\Rightarrow N(0,1).\]
Now note that the MLE  for the functional $\int gw^2$
\[ \hat\theta_n^{MLE}(w^2):=\int_0^1 \hat g^{MLE}(x)w(x)^2dx-\frac1n\sum_{j\ge 1}{\bf 1}\big(\hat g^{MLE}(X_j)=Y_j\big)w(X_j)^2\]
is also unbiased with $\Var(\hat\theta_n^{MLE}(w^2))^{1/2}\lesssim n^{-3/4}$. Since by assumption $\int_0^1\E[\hat g^{MLE}-g]w^2\thicksim n^{-1/2}$ is of larger order, Slutsky's Lemma permits to replace $\int_0^1\E[\hat g^{MLE}-g]w^2$ by $\int_0^1\hat g^{MLE}w^2-\hat\theta_n^{MLE}(w^2)$ in the CLT, which gives the desired self-normalising form.

\bibliography{Literatur}

\end{document}